\documentclass[11pt]{article}

\usepackage{amsfonts}
\usepackage{stmaryrd}
\usepackage{latexsym}
\usepackage{amssymb}
\usepackage{mathrsfs}
\usepackage{dsfont}
\usepackage{amsmath}
\usepackage{amsthm}
\usepackage{cases}
\usepackage{extarrows}
\usepackage{booktabs}
\usepackage{enumerate}

\usepackage[german,english]{babel}

\usepackage[square, comma, sort&compress, numbers]{natbib}

\usepackage{fancyhdr}
\usepackage{geometry}
\usepackage{color}
\usepackage{epstopdf}
\usepackage{graphicx}
\usepackage{subfigure}

\usepackage{enumerate}

\theoremstyle{plain}
\newtheorem{theorem}{Theorem}[section]

\newtheorem{corollary}[theorem]{Corollary}
\newtheorem{lemma}[theorem]{Lemma}
\newtheorem{example}[theorem]{Example}

\newtheorem{remark}[theorem]{Remark}
\newtheorem{definition}[theorem]{Definition}

\numberwithin{equation}{section}

\usepackage{hyperref}

\begin{document}

\title{$\rm P$-generalized determinant over  quaternions  and its applications}

\author{{\ Jiale Gao$^{a}$,    Qing-Wen Wang$^{a,b}$\footnote{Corresponding author:  Qing-Wen Wang; Email: wqw@shu.edu.cn}}
\\ {\small$^{a}$Department of Mathematics, Shanghai University, Shanghai 200444, China}
\\
{{\small $^b$Collaborative Innovation Center for the Marine Artificial Intelligence, Shanghai 200444, China}}
}

\date{}
\maketitle

\begin{abstract}
This paper introduces
the ${\rm P}$-generalized determinant,
a more general class of quaternion determinants.
It unifies  several classical quaternion determinants proposed  between 1922 and 2008, including the Moore determinant, the Xie determinant, the Chen determinant, and the Kyrchei row--column determinants, thereby revealing the essential nature of quaternion determinants as multivariate polynomial-type functions with a fixed ordering of products.
We provide equivalent characterizations of the expansion of the $\mathrm{P}$-generalized determinant along a given row or column, and establish conversion relations between the
$\mathrm{P}_{l}$-  and $\mathrm{P}_{r}$-generalized determinants.
We prove that
 the $\mathrm{P}$-generalized determinant of a quaternion Hermitian matrix  is  a real number  independent of the choice of cycle representation rule.
We further investigate the behavior of the $\mathrm{P}$-generalized determinant under elementary row and column operations on quaternion Hermitian matrices, as well as its linearity in rows and columns.
Finally, using the $\mathrm{P}$-generalized determinant, we derive determinantal representations of the inverse of an invertible quaternion Hermitian matrix and of the unique solution to the quaternion matrix equation $AXB=C$. In particular, this yields  Cramer's rule for the left and right linear systems $Ax=b$ and $xA=b$.
\end{abstract}



\section{Introduction}\label{Introductionsection}
In 1843, the Irish mathematician  Hamilton \cite{hamilton1853lectures} discovered the quaternions
\begin{equation*}
\mathbb{H}=\{q_1+q_2\mathbf{i}+q_3\mathbf{j}+q_4\mathbf{k}\mid \mathbf{i}^2=\mathbf{j}^2=\mathbf{k}^2=\mathbf{ijk}=-1,\;q_1,q_2,q_3,q_4\in\mathbb{R}\},
\end{equation*}
which extend the field of real numbers $\mathbb{R}$ and the field of complex numbers $\mathbb{C}$
to a four-dimensional noncommutative division algebra.
Quaternions have since found significant applications in color image processing \cite{grigoryan2018quaternion},
orbits, aerospace, and virtual reality
\cite{kuipers1999quaternions},
as well as in quantum physics \cite{finkelstein1962foundations,deleo1997quantum}, field theory \cite{adler1995quaternionic}, and relativistic physics \cite{girard2007quaternions}.

In classical linear algebra, the determinant plays a central role: it characterizes invertibility, describes rank, yields the inverse matrix, and underlies Cramer's rule for solving linear systems. Curiously, although Hamilton devoted the remainder of his life to the study of quaternions \cite{hamilton1899elements}, he never addressed the theory of quaternionic determinants. As Dyson \cite{dyson1972quaternion} suggested, this omission likely reflects Hamilton's sense that a satisfactory theory of quaternion determinants was impossible.

Nevertheless, mathematicians did not abandon the study of quaternion determinants.  In 1845,   Cayley \cite{cayley1845quat}  defined  the  first  quaternion determinant, denoted by ${\rm Caydet}$, by expanding recursively along the first column.
However, $\rm Caydet$ found little acceptance among his contemporaries and was ultimately deemed unsatisfactory \cite{aslaksen1996quaternionic}.
What properties should an ideal quaternion determinant satisfy?
 Aslaksen \cite{aslaksen1996quaternionic} proposed that a map $d\colon\mathbb{H}^{n\times n}\to\mathbb{H}$ qualifies as a quaternion determinant precisely when it satisfies the following three axioms:
\begin{description}
\item[Axiom 1] $d(A)=0$ if and only if $A$ is singular.
\item[Axiom 2] $d(AB)=d(A)d(B)$ for any square quaternion  matrices $A,B$.
\item[Axiom 3] If $A_1$ is obtained from $A$ by adding a left-multiple of one row to another, or a right-multiple of one column to another,
     then $d(A_1)=d(A)$.
\end{description}
The $\rm Caydet$  satisfies none of the three axioms
above~\cite{aslaksen1996quaternionic}.
By contrast, in 1920, Study~\cite{study1920zur} used the complex matrix representation $\psi(\cdot)$ of quaternion matrices to define a determinant, denoted by $\operatorname{Sdet}$,  satisfying all three axioms:
\begin{equation*}
\operatorname{Sdet}(A)=\det\nolimits_{\mathbb{C}}\bigl(\psi(A)\bigr),
\end{equation*}
where  $\det\nolimits_{\mathbb{C}}(\cdot)$ denotes the determinant of a complex matrix.
The $\operatorname{Sdet}$ agrees with the $q$-determinant defined by Zhang in his classic survey on quaternions \cite{zhang1997quaternions}.  Moreover, Aslaksen \cite{aslaksen1996quaternionic} employed the real matrix representation $\mu(\cdot)$ of quaternion matrices and proved that $\det\nolimits_{\mathbb{C}}(\mu(A))=({\rm Sdet}(A))^2$.
In 1943, Dieudonn\'{e} \cite{dieudonne1943determinants} redefined the quaternion determinant using commutators from group theory, denoted by ${\rm Ddet}$.
The ${\rm Ddet}$ also satisfies the three axioms, and any quaternion determinant function satisfying these axioms can be expressed as a power of ${\rm Ddet}$ \cite{dieudonne1943determinants}. In particular, ${\rm Sdet}(A)=({\rm Ddet}(A))^2$.

However, Cohen and De Leo \cite{cohen2000quaternionic} observed that the ${\rm Sdet}$ does not truly generalize the complex determinant to the quaternion algebra: the ${\rm det}_{\mathbb{C}}$ takes values throughout $\mathbb{C}$, whereas the ${\rm Sdet}$   takes only nonnegative  real values rather than arbitrary quaternions. In fact, Aslaksen \cite[Theorem~1]{aslaksen1996quaternionic} established a fundamental limitation: any determinant function satisfying Axioms 1--3 must take values in a commutative subset of $\mathbb{H}$. Consequently, no determinant function can simultaneously satisfy all three axioms and take values throughout $\mathbb{H}$.

As early as the 1920s, researchers began exploring determinants that take values throughout the quaternions, at which point relaxing the axiomatic restrictions became inevitable.
In 1922, the American mathematician  Moore, in a short research note \cite{moore1922determinant}, introduced the Moore determinant for quaternion Hermitian matrices, denoted by $\operatorname{Mdet}$.
Let $S_n$ be the symmetric group on the set $[n]:=\{1,2,\ldots,n\}$,
 let  ${\rm sgn}(\sigma)$ be the sign of a permutation $\sigma\in S_n$,
 and let  the  disjoint cycle decomposition of a permutation $\sigma\in S_n$ be written in the usual two-line notation as
\begin{equation}\label{sigmacycdecdisMooreeq}
\sigma=\left(
    \begin{array}{ccccccccccc}
i_{k_1}&\dots&i_{k_1+v_1}&i_{k_2}&\dots&i_{k_2+v_2}&\dots&i_{k_u}&\dots&i_{k_u+v_u}\\
i_{k_1+1}&\dots&i_{k_1}&i_{k_2+1}&\dots&i_{k_2}&\dots&i_{k_u+1}&\dots&i_{k_u}
    \end{array}
  \right),
\end{equation}
where $\sigma(i_{k_d})=i_{k_d+1},
\ldots,
\sigma^{v_d}(i_{k_d})=i_{k_d}$ for all $d\in[u]$.
For a quaternion Hermitian matrix ${A}=[a_{xy}]\in\mathbb{H}^{n\times n}$,
define \cite{moore1922determinant}
\begin{equation}\label{Mooredetdef}
\operatorname{Mdet}({A})=\sum_{\sigma\in S_n}{\rm sgn}(\sigma)
\bigl(a_{i_{k_1},i_{k_1+1}} \cdots a_{i_{k_1+v_1},i_{k_1}}\bigr)   \cdots
\bigl(a_{i_{k_u},i_{k_u+1}}   \cdots a_{i_{k_u+v_u},i_{k_u}}\bigr) ,
\end{equation}
where, for each $\sigma\in S_n$, its cycle decomposition \eqref{sigmacycdecdisMooreeq} is chosen so that $i_{k_t}={\rm min}\{i_{k_t+1},\ldots,i_{k_t+v_t}\}$ for all $t\in[u]$, and $i_{k_1}>i_{k_2}>\cdots>i_{k_u}$.
In 1972, Dyson~\cite{dyson1972quaternion} extended $\operatorname{Mdet}$
to almost Hermitian (self-adjoint) matrices by induction on $n$,
 and posed a far-reaching question: can $\operatorname{Mdet}$ be naturally generalized to arbitrary quaternion matrices?

In 1980, the Chinese mathematician Xie \cite{xie1980centralizable} introduced the determinant of centralizable matrices over an arbitrary division ring, denoted by ${\rm Xdet}$. Observing that quaternion Hermitian matrices are a special case of centralizable matrices,
Xie \cite{xie1980expansion} gave the explicit expansion of ${\rm Xdet}$ for quaternion Hermitian matrices:
\begin{equation}\label{Xiebangjiedetdef}
{\rm Xdet}(A)=\sum_{\sigma\in S_n}{\rm sgn}(\sigma)
\bigl(a_{i_{k_1},i_{k_1+1}} \cdots a_{i_{k_1+v_1},i_{k_1}}\bigr)   \cdots
\bigl(a_{i_{k_u},i_{k_u+1}}   \cdots a_{i_{k_u+v_u},i_{k_u}}\bigr) ,
\end{equation}
where $A=[a_{xy}]\in\mathbb{H}^{n\times n}$ is Hermitian, and the cycle decomposition \eqref{sigmacycdecdisMooreeq} of each $\sigma\in S_n$ is chosen so that $i_{k_t}={\rm min}\{i_{k_t+1},\ldots,i_{k_t+v_t}\}$ for all $t\in[u]$ and $i_{k_1}<i_{k_2}<\cdots<i_{k_u}$.
In 1991, Chen \cite{chen1991definition} defined the determinant of an arbitrary quaternion square matrix, denoted by ${\rm Cdet}$,
\begin{equation}\label{ChenLXdetdef}
\operatorname{Cdet}({A})=\sum_{\sigma\in S_n}{\rm sgn}(\sigma)
\bigl(a_{i_{k_1},i_{k_1+1}} \cdots a_{i_{k_1+v_1},i_{k_1}}\bigr)   \cdots
\bigl(a_{i_{k_u},i_{k_u+1}}   \cdots a_{i_{k_u+v_u},i_{k_u}}\bigr) ,
\end{equation}
where $A=[a_{xy}]\in\mathbb{H}^{n\times n}$, and  the cycle decomposition \eqref{sigmacycdecdisMooreeq} of each $\sigma\in S_n$ is chosen so that $i_{k_t}={\rm max}\{i_{k_t+1},\ldots,i_{k_t+v_t}\}$ for all $t\in[u]$ and $i_{k_1}>i_{k_2}>\cdots>i_{k_u}$.
Furthermore, using ${\rm Cdet}$, Chen defined the double determinant and established Cramer's rule for left and right linear quaternion equations \cite{chen1991definition}, as well as  Hadamard theorem for quaternion matrices \cite{chen1991inverse}.
Meanwhile, in 1991--1992, Gelfand and Retakh \cite{gelfand1991determinants,gelfand1992theory} abandoned the approach of constructing a single determinant and instead defined $n^2$ quasideterminants for $n\times n$ matrices over a noncommutative ring. Quasideterminants  have been successfully applied  to noncommutative geometry, noncommutative algebra, noncommutative symmetric functions, and noncommutative integrable systems \cite{gelfand2005quasideterminants}.

However, none of the above quaternion determinants admits expansion by cofactors along an arbitrary row or column, as the classical complex determinant does, and consequently cannot yield the classical adjugate matrix. It was not until 2008 that Kyrchei \cite{kyrchei2008cramer,kyrchei2012theory} established the row--column determinant theory for quaternion matrices, successfully resolving this problem.  For $A=[a_{xy}]\in\mathbb{H}^{n\times n}$ and $i,j\in[n]$, the $i$-th row determinant is defined as
\begin{equation}\label{Ivanrowdetdef}
{\rm rdet}_i(A)=\sum_{\sigma\in S_n}{\rm sgn}(\sigma)
\bigl(a_{i_{k_1},i_{k_1+1}} \cdots a_{i_{k_1+v_1},i_{k_1}}\bigr)   \cdots
\bigl(a_{i_{k_u},i_{k_u+1}}   \cdots a_{i_{k_u+v_u},i_{k_u}}\bigr) ,
\end{equation}
where the cycle decomposition \eqref{sigmacycdecdisMooreeq} of each $\sigma\in S_n$ is chosen so that $i_{k_1}=i$, $i_{k_t}={\rm min}\{i_{k_t+1},\ldots,i_{k_t+v_t}\}$ for all $t\in[u]\setminus\{1\}$, and $i_{k_2}<\cdots<i_{k_u}$;
and the $j$-th column determinant is defined as
\begin{equation}\label{Ivancolumdetdef}
{\rm cdet}_j(A)=\sum_{\sigma\in S_n}{\rm sgn}(\sigma)
\bigl(a_{i_{k_u},i_{k_u+v_u}}\cdots a_{i_{k_u+1},i_{k_u}}\bigr)
\cdots
\bigl(a_{i_{k_1},i_{k_1+v_1}}\cdots a_{i_{k_1+1},i_{k_1}}\bigr) ,
\end{equation}
where the cycle decomposition \eqref{sigmacycdecdisMooreeq} of each
$\sigma\in S_n$ is chosen so that $i_{k_1}=j$, with the remaining cycles normalized as in $\operatorname{rdet}_j(A)$.
The original idea of the row--column determinant traces back to \cite{kirchej1998fractional}. Since then, this theory has been widely applied to explicit determinantal representations of various generalized inverses, as well as to the solution of two-sided  matrix equations, generalized Sylvester  equations, Lyapunov-type equations, and differential equations over quaternions;  see the monograph \cite{Kyrchei2026} published in 2026.

 By comparing the definitions of ${\rm Mdet}$ \eqref{Mooredetdef}, ${\rm Xdet}$ \eqref{Xiebangjiedetdef}, ${\rm Cdet}$ \eqref{ChenLXdetdef}, ${\rm rdet}_i$ \eqref{Ivanrowdetdef}, and ${\rm cdet}_j$ \eqref{Ivancolumdetdef}, we observe a common feature:
 all five classical quaternion determinants are, in essence, multivariate polynomial-type functions with a fixed ordering of products.
 A natural question arises: what is the essential structure that must be preserved among such quaternion determinants,
 and which choices are inessential?
 Motivated by this, we introduce a more general class of quaternion determinants, termed the {$\rm P$-generalized determinant}, which unifies the five classical determinants above while retaining most of their favorable properties. The main contributions of this paper are as follows:
\begin{enumerate}[(1)]
  \item By formulating left-ordered and right-ordered cycle representation rules for permutations, we define $\mathrm{P}_{l}$-   and $\mathrm{P}_{r}$-generalized determinants, collectively referred to as the $\mathrm{P}$-generalized determinants.

  \item
  We give an equivalent characterization of the expansion of the $\mathrm{P}$-generalized determinant along a given row or column.
  We  then study  its basic properties,
  such as factoring out a common scalar factor from a row  column and additivity with respect to a row or a column.
  Conversion relations between the $\mathrm{P}_{l}$- and $\mathrm{P}_{r}$-generalized  determinants are established.

  \item We reveal the invariance of the $\mathrm{P}$-generalized determinant for a quaternion Hermitian matrix, namely that its value is a real  number   independent of the choice of cycle representation rule.

  \item We investigate the $\mathrm{P}$-generalized determinant when a row (column) of a quaternion Hermitian matrix is replaced by another,
as well as under one-sided elementary row (column) additions and scalings.  We also discuss the connections   between the $\mathrm{P}$-generalized determinant  and Axioms 1--3.

  \item
  Applying the \({\rm P}\)-generalized determinant, we give the determinantal representation of the inverse of an invertible quaternion Hermitian matrix, and define the \({\rm P}\)-adjoint matrix.
  We then apply it to the unique solution of the classical quaternion matrix equation \(AXB=C\), thereby deriving Cramer's rule for the left and right quaternion linear systems \(Ax=b\) and \(xA=b\).
\end{enumerate}

The remainder of this paper is organized as follows. Section \ref{Preliminariessect} recalls the notation, definitions, and lemmas.  Section \ref{Pgenrdetsdefsection} presents the definition of the $\mathrm{P}$-generalized determinant, and its basic properties.
 Section \ref{PgenrdetHersection} investigates the $\mathrm{P}$-generalized determinant for quaternion Hermitian matrices.  Section \ref{appPgenrdetssection} applies the $\mathrm{P}$-generalized determinant to derive the componentwise representation  of the inverse of an invertible quaternion Hermitian matrix and its applications to solving quaternion matrix equations and linear systems.
Conclusions are stated in Section~\ref{conclusection}.

\section{Preliminaries}\label{Preliminariessect}
This section introduces the necessary notation, definitions, and lemmas used throughout the paper.
Let $[n]=\{1,2,\ldots,n\}$ for a positive integer $n$, and let $\mathcal{Q}([n])=\{Q\subseteq[n]\mid Q\neq\emptyset\}$ be the set of all nonempty subsets of $[n]$. Let $S_n$ be the symmetric group on $[n]$.
For a permutation $\sigma\in S_n$, let ${\rm sgn}(\sigma)$ denote the sign of $\sigma$, and ${\rm sgn}(\sigma)=n-r$ if $\sigma$ decomposes into $r$ disjoint cycles (including $1$-cycles); and let $\operatorname{cyc}(\sigma)$ denote the set of all disjoint cycles in the cycle decomposition of $\sigma$.
For a cycle $C\in S_n$, let $\operatorname{supp}(C)=\{i\in[n]\mid C(i)\neq i\}\subseteq[n]$ be the support of $C$, and let $|\operatorname{supp}(C)|$ be the length of $C$.
  A strict total order on a nonempty set is a binary relation that is irreflexive, transitive, and total (any two distinct elements are comparable).

For a quaternion $q=q_0+q_1\mathbf{i}+q_2\mathbf{j}+q_3\mathbf{k}$ with $q_0,q_1,q_2,q_3\in\mathbb{R}$, let
\begin{equation*}
  \overline{q}=q_0-q_1\mathbf{i}-q_2\mathbf{j}-q_3\mathbf{k},\quad
  {\rm tr}(q)=q+\overline{q}=2q_0\in\mathbb{R},\quad
  |q|=\sqrt{q\overline{q}}=\sqrt{q_0^2+q_1^2+q_2^2+q_3^2}
\end{equation*}
denote   the conjugate,  trace, and norm of $q$, respectively.
Let $\mathbb{H}^{m\times n}$ denote the set of all $m\times n$ quaternion matrices. Let $0$ denote the zero matrix of appropriate size and $I_n$ the $n\times n$ identity matrix.  For $A=[a_{xy}]\in\mathbb{H}^{m\times n}$,
 let
 \begin{equation*}
   \overline{A}=[\overline{a_{xy}}]\in\mathbb{H}^{m\times n}, \quad A^{T}=[a_{yx}]\in\mathbb{H}^{n\times m},\quad A^{*}=\overline{A^{T}}=[\overline{a_{yx}}]\in\mathbb{H}^{n\times m}
 \end{equation*}
 denote the conjugate, transpose, and conjugate transpose of $A$, respectively.  Let $P_{ij}(b)$ denote the elementary addition matrix, obtained from $I_n$ by adding a left multiple $b\in\mathbb{H}$ of the $j$-th row to the $i$-th row;
 let $P_i(b)$ denote the elementary scaling matrix, obtained from $I_n$ by left-multiplying the $i$-th row by $b\in\mathbb{H}$;
 let $P_{ij}$ denote the elementary interchange matrix, obtained from $I_n$ by interchanging the $i$-th and $j$-th rows.
  For $A\in\mathbb{H}^{m\times n}$, let $A^{\alpha}_{\beta}\in\mathbb{H}^{m_1\times n_1}$ denote the submatrix of $A$ consisting of the rows indexed by $\alpha=\{\alpha_1,\ldots,\alpha_{m_1}\}\subseteq[m]$ and the columns indexed by $\beta=\{\beta_1,\ldots,\beta_{n_1}\}\subseteq[n]$;
 let $A_{i\cdot}(b)\in\mathbb{H}^{m\times n}$ denote the matrix obtained from $A$ by replacing the $i$-th row with the row vector $b\in\mathbb{H}^{1\times n}$; let $A_{\cdot j}(c)\in\mathbb{H}^{m\times n}$ denote the matrix obtained from $A$ by replacing the $j$-th column with the column vector $c\in\mathbb{H}^{m\times 1}$; let $A^{i,j}\in\mathbb{H}^{(m-1)\times(n-1)}$ denote the matrix obtained from $A$ by deleting the $i$-th row and the $j$-th column;
 and let $a_{i\cdot}\in\mathbb{H}^{1\times n}$ and $a_{\cdot j}\in\mathbb{H}^{m\times 1}$ denote the $i$-th row and the $j$-th column of $A$, respectively.

We use the following convention throughout this paper: for any matrix and its submatrices, the row and column indices are taken from the same index set; that is, a submatrix inherits the row and column indices of the original matrix and is not reindexed.
For example, let $A\in\mathbb{H}^{3\times 4}$ have row index set $\{1,2,3\}$ and column index set $\{1,2,3,4\}$. Take $\alpha=\{2,3\}$ and $\beta=\{1,4\}$, and write
$
A_1:=A^{\alpha}_{\beta}\in\mathbb{H}^{2\times 2}.
$
 Then the row index set of $A_1$ is $\{2,3\}$, and its column index set is $\{1,4\}$.
 Consequently, in the natural increasing order of the indices in $\alpha$ and $\beta$, the rows of $A_1$ indexed by $2$ and $3$ are its first and second rows under the local numbering, respectively; the columns of $A_1$ indexed by $1$ and $4$ are its first and second columns under the local numbering, respectively.

\begin{lemma}\cite[Lemma 3.1]{kyrchei2008cramer}\label{kyrchei2008cramerLemma31}
Let  $h_1,h_2,...,h_n\in\mathbb{H}$, and  let $T_n$ denote the sum over all ordered products of $n$ factors, where for each index $i\in\{1,2,\dots,n\}$, the $i$-th factor is chosen to be either $h_i$ or $\bar h_i$, i.e.,
$
T_n = \sum_{q_i \in \{ h_i, \overline{h_i}\},  ~i=1,2,...,n}  q_1q_2...q_n.
$
Then,  $T_n = {\rm tr}(h_1){\rm tr}(h_2)\cdots {\rm tr}(h_n) \in \mathbb{R}.$

\end{lemma}

\section{\(\rm P\)-generalized determinant}
\label{Pgenrdetsdefsection}

This section provides the rigorous mathematical definition of the $\mathrm{P}$-generalized determinant and investigates its basic properties. To define it, we first establish cycle representation rules for permutations in the symmetric group $S_n$.

\begin{definition}
Let $S_n$ be the symmetric group on the set $[n]$.
\begin{enumerate}[(1)]
  \item
Define a function
\begin{equation*}
  {\mathrm{sel}}:\mathcal{Q}([n])\to[n],
\end{equation*}
such that ${\mathrm{sel}}(Q)\in Q$ for every nonempty $Q\in\mathcal{Q}([n])$.
We call this function a selector on $\mathcal{Q}([n])$ and refer to ${\mathrm{sel}}(Q)$ as the selected element of $Q$.

\item
Let $\vartriangle$ be a strict total order on $[n]$, and let $  {\mathrm{sel}}$ be a selector on $\mathcal{Q}([n])$.  For
  $\sigma\in S_n$,
define a strict total order $\mathrel{\curlywedge}_{\vartriangle}$  on $\mathrm{cyc}(\sigma)$ by
\begin{equation*}
  C\,\mathrel{\curlywedge}_{\vartriangle}\,D \iff
  {\rm sel}(\mathrm{supp}(C))\,  \vartriangle  \,{\rm sel}(\mathrm{supp}(D)),
\end{equation*}
where $C,D\in\mathrm{cyc}(\sigma)$ are two distinct cycles.
 To emphasize its dependence on $\mathrm{cyc}(\sigma)$, $\vartriangle$ and $  {\mathrm{sel}}$, we also write this order as ${\curlywedge}_{\vartriangle}\big(\mathrm{cyc}(\sigma), \vartriangle,  {\mathrm{sel}}\big)$.

\end{enumerate}
\end{definition}

\begin{definition}
Let $S_n$ be the symmetric group on the set $[n]$.
\begin{enumerate}[(1)]
\item
 Let $ {\mathrm{first}}$ be a selector on  $\mathcal{Q}([n])$,
 let  $\triangleleft$ be a strict total order on $[n]$, and let
$\prec_{\triangleleft}:= \prec_{\triangleleft}\big(\mathrm{cyc}(\sigma), \triangleleft, {\mathrm{first}}  \big) $ for $\sigma\in S_n$.
 Define  the left-order representation of the cycle $C\in\mathrm{cyc}(\sigma)$  based on the selector  ${\rm first}$ by
\begin{equation*}
 C_{\mathfrak{l}}= \big(c,
\sigma(c), \sigma^2(c), \dots ,
\sigma^{|\mathrm{supp}(C)|-1}(c)\big),
\end{equation*}
where $c={\mathrm{first}}(\mathrm{supp}(C))$.
 Furthermore, define the left-order cycle representation of the permutation $\sigma$ based on the selector ${\rm first}$ and the strict total order $\prec_{{\triangleleft}} $ by
\begin{equation}\label{cicdecsigmapereq}
  \sigma_l={(C_1)}_{\mathfrak{l}} (C_2)_{\mathfrak{l}}
\dots(C_k)_{\mathfrak{l}},
\end{equation}
where $C_1,\dots,C_k$ are all cycles of $\sigma$, arranged in increasing order
with respect to $\prec_{\triangleleft}$, i.e.,
\begin{equation*}
C_1\prec_{\triangleleft} C_2\prec_{\triangleleft}\cdots\prec_{\triangleleft} C_k.
\end{equation*}
 The rule that determines the left-order cycle representation  \eqref{cicdecsigmapereq}  of every permutation in $S_n$, based on the selector ${\rm first}$ and the strict total order $\prec_{{\triangleleft}}$, is called a left-order cycle representation rule. It is denoted by ${\rm P}_l(S_n,{\rm first},\prec_{{\triangleleft}})$ and abbreviated as ${\rm P}_l$.
    In particular, if, for every $\sigma\in S_n$, the selected element of the cycle $(C_1)_\mathfrak{l}$ in \eqref{cicdecsigmapereq} is fixed at a prescribed element $i\in[n]$, then this rule is denoted by ${\rm P}_{l\text{-}i}$.

\item
 Let $ {\mathrm{last}}$ be a selector on  $\mathcal{Q}([n])$,
let  $\triangleright$  be a  strict total order   on $[n]$,
 and let $\succ_{{\triangleright}}:=
 \succ_{{\triangleright}}
 \big(\mathrm{cyc}(\sigma),
 \triangleright, {\mathrm{last}}  \big) $ for $\sigma\in S_n$.
Define the right-order representation of the cycle $D \in\mathrm{cyc}(\sigma)$ based on the selector  ${\rm last}$ by
\begin{equation*}
 D_{\mathfrak{r}}= \big(
\sigma^{|\mathrm{supp}(D)|-1}(d) ,\dots , \sigma^2(d) , \sigma(d) , d  \big),
\end{equation*}
where $d={\mathrm{last}}(\mathrm{supp}(D))$.
 Furthermore, define the right-order cycle representation of the permutation $\sigma$ based on the selector ${\rm last}$ and the strict total order $\succ_{{\triangleright}}$ by
\begin{equation}\label{sigmadecrighteq}
  \sigma_r=(D_k)_{\mathfrak{r}} \dots (D_2)_{\mathfrak{r}} {(D_1)}_{\mathfrak{r}} ,
\end{equation}
where $D_1,\dots,D_k $ are all cycles of $\sigma$,
 arranged in decreasing order with respect to $\succ_{{\triangleright}}$, i.e.,
 \begin{equation*}
 D_k\succ_{{\triangleright}}  \ldots \succ_{{\triangleright}}
 D_2\succ_{{\triangleright}} D_1.
\end{equation*}
The rule that determines the right-order cycle representation \eqref{sigmadecrighteq} of every permutation in $S_n$, based on the selector ${\rm last}$ and the strict total order $\succ_{{\triangleright}}$, is called a right-order cycle representation rule. It is denoted by ${\rm P}_r(S_n,{\rm last},\succ_{{\triangleright}})$ and abbreviated as ${\rm P}_r$.
In particular, if, for every $\sigma\in S_n$, the selected element of the cycle $(D_1)_\mathfrak{r}$ in \eqref{sigmadecrighteq} is fixed at a prescribed element $j\in[n]$, then this special rule is denoted by ${\rm P}_{r\text{-}j}$.

\item
The left-order cycle representation rule ${\rm P}_l$ and the right-order cycle representation rule ${\rm P}_r$ are collectively referred to as cycle representation rules (or simply rules) and denoted generically by ${\rm P}$.

\end{enumerate}

\end{definition}

\begin{definition}\label{Pgdetdeinition}
Let $S_n$ be the symmetric group on  $[n]$, and  $A=[a_{ij}]\in\mathbb{H}^{n\times n}$.
\begin{enumerate}[(1)]
  \item
   Let ${\rm P}_l={\rm P}_l( S_n, {\rm first}, \prec_{{\triangleleft}})$ be a left-order cycle representation rule.
Under the rule ${\rm P}_l$, suppose that the left-order cycle representation  of a permutation $\sigma\in S_n$ is
\begin{equation}\label{sigmmaldeceq}
  \sigma_l=(i_{k_1},i_{k_1+1},...,i_{k_1+v_1})(i_{k_2},i_{k_2+1},...,i_{k_2+v_2})...
  (i_{k_u},i_{k_u+1},...,i_{k_u+v_u}),
\end{equation}
where $i_{k_1},...,i_{k_1+v_1},...,i_{k_u},...,i_{k_u+v_u}\in[n]$ are fixed and mutually distinct. Let
\begin{equation}\label{fasiggmaldeceq}
  f_A(\sigma_l)=
  (a_{i_{k_1},i_{k_1+1}}a_{i_{k_1+1},i_{k_1+2}}...a_{i_{k_1+v_1},i_{k_1}})
\cdots
  (a_{i_{k_u},i_{k_u+1}}a_{i_{k_u+1},i_{k_u+2}}...a_{i_{k_u+v_u},i_{k_u}}).
\end{equation}
Then the ${\rm P}_l$-generalized determinant of $A$, denoted by
$\mathrm{GDet}_{{\rm P}_l}(A)$, is defined by
\begin{equation*}
  \mathrm{GDet}_{{\rm P}_l}(A)=
  \sum_{\sigma\in S_n}\mathrm{sgn}(\sigma)f_A(\sigma_l).
\end{equation*}

  \item
  Let ${\rm P}_r={\rm P}_r( S_n, {\rm last}, \succ_{{\triangleright}})$ be a right-order cycle representation rule.
      Under the rule ${\rm P}_{r}$, suppose that the right-order cycle representation  of a permutation $\sigma\in S_n$ is
\begin{equation*}
  \sigma_r=(j_{k_u+v_u},...,j_{k_u+1},j_{k_u})...(j_{k_2+v_2},...,j_{k_2+1},j_{k_2})(j_{k_1+v_1},...,j_{k_1+1},j_{k_1}),
\end{equation*}
where $j_{k_u+v_u},...,j_{k_u},...,j_{k_1+v_1},...,j_{k_1}\in[n]$ are fixed and mutually distinct.
Let
\begin{equation*}
  f_A( \sigma_r)= (a_{j_{k_u},j_{k_u+ v_u}}...a_{j_{k_u+2},j_{k_u+1}}a_{j_{k_u+1},j_{k_u}})\cdots
(a_{j_{k_1},j_{k_1+ v_1}}...a_{j_{k_1+2},j_{k_1+1}}a_{j_{k_1+1},j_{k_1}}).
\end{equation*}
Then the ${\rm P}_{r}$-generalized determinant of $A$,
denoted by
$ {\rm GDet}_{{\rm P}_r} (A)$,
is defined by
\begin{equation*}
 {\rm GDet}_{{\rm P}_r}
 (A)=\sum\limits_{\sigma\in S_n} {\rm sgn}(\sigma)   f_A( \sigma_r).
\end{equation*}

\item The ${\rm P}_l$-generalized determinant ${\rm GDet}_{{\rm P}_l}(A)$ and the ${\rm P}_r$-generalized determinant ${\rm GDet}_{{\rm P}_r}(A)$  are collectively called the ${\rm P}$-generalized determinants, denoted by ${\rm GDet}_{{\rm P}}(A)$.

\end{enumerate}

\end{definition}

\begin{remark}\label{PGdetspecaseremark}
Under the notation of Definition~\ref{Pgdetdeinition},
 by choosing specific cycle representation rules, one recovers the classical quaternion determinants, as follows:
\begin{enumerate}[(1)]
\item Let ${\rm first}(Q)=\min(Q)$ for $Q\in \mathcal{Q}([n])$, and let $\triangleleft$ be the strict total order on $[n]$ defined by
\[
n\triangleleft n-1\triangleleft \ldots \triangleleft 2\triangleleft 1.
\]
Then, under the rule ${\rm P}_l(S_n, {\rm first}, \prec_{{\triangleleft}})$, the ${\rm P}_l$-generalized determinant coincides with the Moore determinant for a Hermitian matrix $A\in\mathbb{H}^{n\times n}$, i.e.,
\begin{equation*}
  {\rm GDet}_{{\rm P}_l}(A)={\rm Mdet}(A).
\end{equation*}

\item
   Let ${\rm first}(Q)=\max(Q)$ for $Q\in \mathcal{Q}([n])$, and let $\triangleleft$ be the strict total order on $[n]$ defined by
\[
n \triangleleft n-1\triangleleft \ldots \triangleleft 2\triangleleft 1.
\]
Then, under the rule ${\rm P}_l(S_n, {\rm first}, \prec_{{\triangleleft}})$, the ${\rm P}_l$-generalized determinant coincides with the Chen determinant, i.e.,
  \begin{equation*}
    {\rm GDet}_{{\rm P}_l}(A)={\rm Cdet}(A).
  \end{equation*}

\item Let ${\rm first}(Q)=\min(Q)$ for $Q\in \mathcal{Q}([n])$, and let $\triangleleft$ be the strict total order on $[n]$ defined by
\[
1\triangleleft 2\triangleleft \ldots n-1 \triangleleft n.
\]
Then, under the rule ${\rm P}_l(S_n,{\rm first},\prec_{{\triangleleft}})$, the ${\rm P}_l$-generalized determinant coincides with the Xie determinant for a Hermitian matrix $A\in\mathbb{H}^{n\times n}$, i.e.,
\[
{\rm GDet}_{{\rm P}_l}(A)={\rm Xdet}(A).
\]

\item\label{PGdetspecaseremarkitem4}
Let  $i\in[n]$ be a  row index,  let
\begin{equation*}
  {\rm first}(Q)= \begin{cases}
i, & i\in Q,\\
\min(Q), & i\notin Q
\end{cases}
\end{equation*}
 for $Q\in \mathcal{Q}([n])$, and let $\triangleleft$ be the strict total order on $[n]$ defined by
\[
i\triangleleft i_1 \triangleleft  i_2 \triangleleft \ldots \triangleleft i_{n-1},
\]
where $i_1, \ldots, i_{n-1} \in [n]\setminus \{i\}$ satisfy
$
   i_1 < i_2 < \ldots < i_{n-1}.
$
Under the rule ${\rm P}_{l\text{-}i}(S_n,{\rm first},\prec_{{\triangleleft}})$, the ${\rm P}_{l\text{-}i}$-generalized determinant coincides with the $i$-th row determinant, i.e.,
\[
{\rm GDet}_{{\rm P}_{l\text{-}i}}(A)={\rm rdet}_i(A).
\]
Note that  $ {\rm rdet}_1(A)={\rm Xdet}(A) $.

\item\label{PGdetspecaseremarkitem5}
Let  $j\in[n]$ be a  column  index,   let
  \begin{equation*}
    {\rm last}(Q)= \begin{cases}
j, & j\in Q,\\
\min(Q), & j\notin Q
\end{cases}
  \end{equation*}
   for $Q\in \mathcal{Q}([n])$,
   and let $\triangleright$ be the strict total order on $[n]$ defined by
\[
j_{n-1}\triangleright  \ldots \triangleright  j_2 \triangleright  j_1 \triangleright j,
\]
where $j_{n-1}, \ldots, j_1 \in [n]\setminus \{j\}$ satisfy
$
j_{n-1}>\ldots>j_2>j_1.
$
Under the rule ${\rm P}_{r\text{-}j}(S_n,{\rm last},\succ_{{\triangleright}})$, the ${\rm P}_{r\text{-}j}$-generalized determinant coincides with the $j$-th column determinant, i.e.,
\begin{equation*}
  {\rm GDet}_{{\rm P}_{r\text{-}j}}(A)={\rm cdet}_j(A).
\end{equation*}
\end{enumerate}
Given the generality of the rule ${\rm P}$, further
imposing different restrictions on it yields several definitions of quaternion determinants that have not yet been studied in depth, as follows:
\begin{enumerate}[(i)]

\item
 Let ${\rm first}(Q)=\max(Q)$ for $Q\in \mathcal{Q}([n])$, and let $\triangleleft$ be the strict total order on $[n]$ defined by
\[
1\triangleleft 2\triangleleft \ldots \triangleleft  n-1 \triangleleft n.
\]
Then, ${\rm GDet}_{{\rm P}_l}(A)$ is called the dual Moore determinant, denoted by ${\rm DMdet}(A)$.

\item  Let ${\rm last}(Q)=\min(Q)$ for $Q\in \mathcal{Q}([n])$, and let $\triangleright$ be the strict total order on $[n]$ defined by
\[
1\triangleright 2\triangleright \ldots \triangleright  n-1 \triangleright n.
\]
Then, ${\rm GDet}_{{\rm P}_{r}}(A)$ is called the column Moore determinant, denoted by ${\rm CMdet}(A)$.

\item
   Let ${\rm last}(Q)=\max(Q)$ for $Q\in \mathcal{Q}([n])$, and let $\triangleright$ be the strict total order on $[n]$ defined by
\[
n\triangleright n-1 \triangleright \ldots \triangleright  2 \triangleright 1.
\]
Then,
  ${\rm GDet}_{{\rm P}_{r}}(A)$ is called the dual column Moore determinant, denoted by ${\rm DCMdet}(A)$.

\item
Let  $i\in[n]$ be a  row index,  let
\begin{equation*}
  {\rm first}(Q)= \begin{cases}
i, & i\in Q,\\
\max(Q), & i\notin Q
\end{cases}
\end{equation*}
 for $Q\in \mathcal{Q}([n])$, and let $\triangleleft$ be the strict total order on $[n]$ defined by
\[
i\triangleleft i_1 \triangleleft  i_2 \triangleleft \ldots \triangleleft i_{n-1},
\]
where $i_1, \ldots, i_{n-1} \in [n]\setminus \{i\}$ satisfy
$
  i_1  > i_2 > \ldots >  i_{n-1}.
$
Then, ${\rm GDet}_{{\rm P}_{l\text{-}i}}(A)$ is called the dual $i$-th row determinant, denoted by ${\rm Drdet}_i(A)$. Moreover, ${\rm Drdet}_n(A)={\rm Cdet}(A)$.

\item
      Let ${\rm last}(Q)=\max(Q)$ for $Q\in \mathcal{Q}([n])$, and let $\triangleright$ be the strict total order on $[n]$ defined by
\[
1 \triangleright  2  \triangleright \ldots \triangleright  n- 1\triangleright n .
\]
  Then, ${\rm GDet}_{{\rm P}_{r}}(A)$ is called the column Chen determinant, denoted by ${\rm CCdet}(A)$.

\item
Let  $j\in[n]$ be a  column  index,   let
  \begin{equation*}
    {\rm last}(Q)= \begin{cases}
j, & j\in Q,\\
\max(Q), & j\notin Q
\end{cases}
  \end{equation*}
   for $Q\in \mathcal{Q}([n])$,
   and let $\triangleright$ be the strict total order on $[n]$ defined by
\[
j_{n-1}\triangleright  \ldots \triangleright  j_2 \triangleright  j_1 \triangleright j,
\]
where $j_{n-1}, \ldots, j_1 \in [n]\setminus \{j\}$ satisfy
$
j_{n-1}<\ldots<j_2<j_1,
$
    then ${\rm GDet}_{{\rm P}_{r\text{-}j}}(A)$ is called the dual $j$-th column determinant, denoted by ${\rm Dcdet}_j(A)$. Moreover, ${\rm Dcdet}_n(A)={\rm CCdet}(A)$.

\end{enumerate}

\end{remark}

\begin{example}
Let
\[
A=\begin{bmatrix}
\mathbf{i} & \mathbf{k}\\
\mathbf{i} & \mathbf{j}
\end{bmatrix}\in \mathbb{H}^{2\times 2}.
\]
Under the cycle representation rules listed in Remark \ref{PGdetspecaseremark}, we compute the corresponding determinants of \(A\).
The results are as follows:
\[
\begin{aligned}
&\mathrm{Mdet}(A)=a_{22}a_{11}-a_{12}a_{21}=-\mathbf{k}-\mathbf{j},
&&\mathrm{DMdet}(A)=a_{11}a_{22}-a_{21}a_{12}=\mathbf{k}+\mathbf{j},\\
&\mathrm{CMdet}(A)=a_{11}a_{22}-a_{12}a_{21}=\mathbf{k}-\mathbf{j},
&&\mathrm{DCMdet}(A)=a_{22}a_{11}-a_{21}a_{12}=-\mathbf{k}+\mathbf{j},\\
&\mathrm{rdet}_1(A)= \mathrm{Xdet}(A)=a_{11}a_{22}-a_{12}a_{21}=\mathbf{k}-\mathbf{j},
&&\mathrm{rdet}_2(A)=a_{22}a_{11}-a_{21}a_{12}=-\mathbf{k}+\mathbf{j},\\
&\mathrm{cdet}_1(A)=a_{22}a_{11}-a_{12}a_{21}=-\mathbf{k}-\mathbf{j},
&&\mathrm{cdet}_2(A)=a_{11}a_{22}-a_{21}a_{12}=\mathbf{k}+\mathbf{j},\\
&\mathrm{Drdet}_1(A)=a_{11}a_{22}-a_{12}a_{21}=\mathbf{k}-\mathbf{j},
&&\mathrm{Drdet}_2(A)=\mathrm{Cdet}(A)=a_{22}a_{11}-a_{21}a_{12}=-\mathbf{k}+\mathbf{j},\\
&\mathrm{Dcdet}_1(A)=a_{22}a_{11}-a_{12}a_{21}=-\mathbf{k}-\mathbf{j},
&&\mathrm{Dcdet}_2(A)=\mathrm{CCdet}(A)=a_{11}a_{22}-a_{21}a_{12}=\mathbf{k}+\mathbf{j}.
\end{aligned}
\]
\end{example}

Next,
analogously to the row- or column-expansion of classical complex  determinants,
we obtain equivalent characterizations of ${\rm GDet}_{{\rm P}_{l\text{-}i}}(A)$
and ${\rm GDet}_{{\rm P}_{r\text{-}j}}(A)$.

\begin{theorem}\label{cofatortheorem}
Let $A=[a_{xy}]\in\mathbb{H}^{n\times n}$ for all  $x,y\in[n]$.
\begin{enumerate}[(1)]
 \item\label{cofatortheoremitem1}

 Let $i\in[n]$, and let $\mathrm{P}_{l\text{-}i}(S_n,\mathrm{first},\prec_{\triangleleft})$ be a left-order cycle representation rule. Let $S_{n_1}$ be the symmetric group on the set $[n]_i:=[n]\setminus\{i\}$.
Suppose the following:
\begin{enumerate}[(i)]
    \item Let
    \begin{equation*}
        \operatorname{P}_{l_1}=\operatorname{P}_{l_1}(S_{n_1}, \operatorname{first}_1, \prec_{\triangleleft_1})
    \end{equation*}
    be a left-order cycle representation rule,
    where $\operatorname{first}_1(Q)=\operatorname{first}(Q)$ for all $Q\in\mathcal{Q}([n]_i)$, and $\triangleleft_1$ is a strict total order on $[n]_i$ defined by
    \[
    t_1\triangleleft_1 t_2\triangleleft_1\cdots\triangleleft_1 t_{n-1}
    \]
   for  $t_1,t_2,\dots,t_{n-1}\in[n]_i $ satisfying $t_1\triangleleft t_2\triangleleft\cdots\triangleleft t_{n-1}$.

    \item\label{cofatortheoremitem1ii} For each $j\in[n]_i$, let
    \begin{equation}\label{cofatortheoremplgj}
        \operatorname{P}_{l\text{-}j}=\operatorname{P}_{l\text{-}j}(S_{n_1},
        \operatorname{first}_2,\prec_{\triangleleft_2})
    \end{equation}
    be a left-order cycle representation rule,
    where
    \[
    \operatorname{first}_2(Q)=
    \begin{cases}
        j, & j\in\operatorname{supp}(Q),\\
        \operatorname{first}(Q), & j\notin\operatorname{supp}(Q)
    \end{cases}
    \]
    for all $Q\in\mathcal{Q}([n]_i)$, and $\triangleleft_2$ is a strict total order on $[n]_i$ defined by
    \[
    j\triangleleft_2 d_1\triangleleft_2 d_2\triangleleft_2\cdots\triangleleft_2 d_{n-2}
    \]
   for  $d_1,d_2,\dots,d_{n-2}\in[n]_i \setminus \{j\}$ satisfying  $d_1\triangleleft d_2\triangleleft\cdots\triangleleft d_{n-2}$.
\end{enumerate}
Then
\begin{equation}\label{cofatortheoremitem1eq}
\mathrm{GDet}_{\mathrm{P}_{l\text{-}i}}(A)=a_{i1}R_{i1}+a_{i2}R_{i2}+\cdots+a_{in}R_{in},
\end{equation}
where
\begin{equation}\label{cofatortheoremitem1Rijeq}
R_{ij}=
\begin{cases}
\mathrm{GDet}_{\operatorname{P}_{l_1}}(A^{i,i}), & j=i,\\
-\mathrm{GDet}_{\operatorname{P}_{l\text{-}j}}\left((A_{\cdot j}(a_{\cdot i}))^{i,i}\right), & j \in  [n]_i.
\end{cases}
\end{equation}

\item\label{cofatortheoremitem2} Let $j \in [n]$, and let ${\rm P}_{r\text{-}j}(S_n, {\rm last}, \succ_{{\triangleright}})$ be a right-order cycle representation rule.
    Let $S_{n_1}$ be the symmetric group on the set $[n]_j=[n]\setminus\{j\}$.
 Suppose the following:
\begin{enumerate}[(i)]
  \item  Let
  \begin{equation*}
    \operatorname{P}_{r_1}=\operatorname{P}_{r_1}(S_{n_1}, \operatorname{last}_1, \succ_{\triangleright_1})
  \end{equation*}
   be a right-order cycle representation rule,
   where $\operatorname{last}_1(Q)=\operatorname{last}(Q)$ for all $Q\in\mathcal{Q}([n]_j)$, and $\triangleright_1$ is a strict total order on $[n]_j$ defined by
  \[
  t_{n-1}\triangleright_1 \ldots \triangleright_1 t_2\triangleright_1 t_1
  \]
for $t_{n-1},\ldots,t_1\in[n]_j$ satisfying
  $t_{n-1}\triangleright \ldots \triangleright t_2\triangleright t_1$.

  \item\label{cofatortheoremitem2ii}  For each $i \in  [n]_j$, let
  \begin{equation}\label{cofatortheoremprgi}
    \operatorname{P}_{r\text{-}i}=\operatorname{P}
  _{r\text{-}i}(S_{n_1},\operatorname{last}_2,\succ_{\triangleright_2})
  \end{equation}
  be a right-order cycle representation rule,
  where
  \[
  \operatorname{last}_2(Q)=
  \begin{cases}
  i, & i\in\operatorname{supp}(Q),\\
  \operatorname{last}(Q), & i\notin\operatorname{supp}(Q)
  \end{cases}
  \]
  for $Q\in\mathcal{Q}([n]_j)$, and $\triangleright_2$ is a strict total order on $[n]_j$ defined by
  \[
  d_{n-2}\triangleright_2 \ldots \triangleright_2 d_2\triangleright_2 d_1\triangleright_2 i
  \]
for  $d_{n-2},\ldots,d_2, d_1\in[n]_j\setminus \{i\}$ satisfying
  $d_{n-2}\triangleright \ldots \triangleright d_2\triangleright d_1$.
\end{enumerate}
Then
\begin{equation}\label{cofatortheoremitem2eq}
  {\rm GDet}_{{\rm P}_{r\text{-}j}}(A)
  =L_{1j}a_{1j}+L_{2j}a_{2j}+\ldots+L_{nj}a_{nj},
\end{equation}
where
\begin{equation}\label{cofatortheoremitem2Lijeq}
  L_{ij}=\begin{cases}
    {\rm GDet}_{{\rm P}_{r_1}}(A^{j,j}), & i=j,\\[1mm]
    -{\rm GDet}_{{\rm P}_{r\text{-}i}}\left((A_{i\cdot}(a_{j\cdot}))^{j,j}\right), & i \in  [n]_j.
  \end{cases}
\end{equation}

\end{enumerate}

\end{theorem}

\begin{proof}
\begin{enumerate}[(1)]
  \item By Definition \ref{Pgdetdeinition},  there exist  $R_{i1}, ..., R_{in} \in \mathbb{H} $ such that
  \begin{equation*}
      {\rm GDet}_{{\rm P}_{l\text{-}i}}(A)=a_{i1}R_{i1}+a_{i2}R_{i2}+\ldots+a_{in}R_{in}.
  \end{equation*}
  Then, we only need to prove that \eqref{cofatortheoremitem1Rijeq} hold.
  First, we prove that
  \begin{equation}\label{cofatortheoremitem1eqfirst}
    R_{ii}=
    {\rm GDet}_{{\rm P}_{l_1}}(A^{i,i}).
  \end{equation}
Indeed, under the rule ${\rm P}_{l\text{-}i}$, let
$\widetilde{S_n}=\{\sigma\in S_n:\sigma_l\text{ starts with }(i)\}$,
that is, if $\sigma\in\widetilde{S_n}$, then
\[
\sigma_l=(i)(i_{k_2},i_{k_2+1},\ldots,i_{k_2+v_2})\ldots(i_{k_u},i_{k_u+1},\ldots,i_{k_u+v_u}).
\]
Then, for $\sigma\in\widetilde{S_n}$, the monomial $f_A(\sigma_l)$ begins on the left with $a_{ii}$. The sum of all such monomials in ${\rm GDet}_{{\rm P}_{l\text{-}i}}(A)$  is
\begin{align*}
  \sum_{{\sigma}\in \widetilde{S_n}} {\rm sgn}(\sigma)f_{A}(\sigma_l) = & \sum_{{\sigma}\in \widetilde{S_n}} (-1)^{n-u} a_{ii}
 \big(a_{i_{k_2},i_{k_2+1}} \ldots a_{i_{k_2+v_2},i_{k_2} }\big) \ldots \big( a_{i_{k_u},i_{k_u+1}} \ldots a_{i_{k_u+v_u}, i_{k_u}}\big) \\
 =& a_{ii}  \left(\sum_{{\sigma}\in \widetilde{S_n}} (-1)^{n-u}
 \big(a_{i_{k_2},i_{k_2+1}} \ldots a_{i_{k_2+v_2},i_{k_2} }\big) \ldots
 \big( a_{i_{k_u},i_{k_u+1}} \ldots a_{i_{k_u+v_u}, i_{k_u}}\big) \right)
 \\
 = &a_{ii} R_{ii}.
\end{align*}
Note that the left-order cycle representations under ${\rm P}_{l\text{-}i}$
and ${\rm P}_{l_1}$ induce a bijection between the permutations in
$\widetilde{S_n}$ and $S_{n_1}$, via the map
 \begin{equation*}
 \Psi :   \widetilde{S_n}  \rightarrow  S_{n_1}  \quad \text{ by } \quad\sigma  \in   \widetilde{S_n}   \mapsto \Psi(  \sigma)  \in S_{n_1},
 \end{equation*}
 where  the left-order cycle  representations of  $\sigma$ and  $\Psi(  \sigma) $ are
  \begin{align*}
  & \sigma_l= (i)(i_{k_2},i_{k_2+1},...,i_{k_2+v_2})...
  (i_{k_u},i_{k_u+1},...,i_{k_u+v_u}),\\
  &  \big(\Psi (  \sigma)\big)_l= (i_{k_2},i_{k_2+1},...,i_{k_2+v_2})...
  (i_{k_u},i_{k_u+1},...,i_{k_u+v_u}),
 \end{align*}
 respectively.  Then,
\begin{align*}
  R_{ii} &=  \sum_{{\sigma}\in \widetilde{S_n}} (-1)^{n-u}
 \big(a_{i_{k_2},i_{k_2+1}} \ldots a_{i_{k_2+v_2},i_{k_2} }\big) \ldots
 \big( a_{i_{k_u},i_{k_u+1}} \ldots a_{i_{k_u+v_u}, i_{k_u}}\big) \\
 &  =   \sum_{ \Psi (  \sigma)\in {S_{n_1}}} (-1)^{n-u}
 \big(a_{i_{k_2},i_{k_2+1}} \ldots a_{i_{k_2+v_2},i_{k_2} }\big) \ldots
 \big( a_{i_{k_u},i_{k_u+1}} \ldots a_{i_{k_u+v_u}, i_{k_u}}\big)
 \\
 &  =   \sum_{ \Psi (  \sigma)\in {S_{n_1}}} (-1)^{(n-1)-(u-1)}
 \big(a_{i_{k_2},i_{k_2+1}} \ldots a_{i_{k_2+v_2},i_{k_2} }\big) \ldots
 \big( a_{i_{k_u},i_{k_u+1}} \ldots a_{i_{k_u+v_u}, i_{k_u}}\big)
 \\ &  =  \sum_{ \Psi (  \sigma) \in {S_{n_1}}}
  {\rm sgn}\big(\Psi (  \sigma)\big)f_{A^{i,i}}
  \left(\big(\Psi (\sigma)\big)_l \right) =
  {\rm GDet}_{{\rm P}_{l_1}}\big(A^{i,i}\big).
\end{align*}
 Next,  for  $j \in [n]_i$,  we prove that
 \begin{equation}\label{cofatortheoremitem1eqsecond}
   R_{ij}= -{\rm GDet}_{{\rm P}_{l\text{-}j}}\left((A_{\cdot j}(a_{\cdot i}))^{i,i}\right).
 \end{equation}
 Indeed,
under the rule ${\rm P}_{l\text{-}i}$,
 let
 $\widehat{S_n}$ be the set of permutations
in $S_n$ whose first left-order cycle is of the form $(i,j,\ldots)$,
that is, if ${\sigma}\in \widehat{S_n}$, then
     \begin{equation*}
       {\sigma}_l =(i,j,i_{k_1+2},..., i_{k_1+v_1} )(i_{k_2},i_{k_2+1},...,i_{k_2+v_2})...
  (i_{k_u},i_{k_u+1},...,i_{k_u+v_u}).
     \end{equation*}
     Then,  for $\sigma\in\widehat{S_n}$,
     the monomial $f_A(\sigma_l)$ begins on the left with $a_{ij}$. The sum of all such monomials in ${\rm GDet}_{{\rm P}_{l\text{-}i}}(A)$ is
       \begin{align*}
   &\sum_{{\sigma}\in \widehat{S_n}} {\rm sgn}(\sigma)f_{A}(\sigma_l)\\
 =& \sum_{{\sigma}\in \widehat{S_n}} (-1)^{n-u} \big(a_{ij} a_{j, i_{k_1+2}},...,  a_{  i_{k_1+v_1},i }  \big)
 \big(a_{i_{k_2},i_{k_2+1}} \ldots a_{i_{k_2+v_2},i_{k_2} }\big) \ldots \big( a_{i_{k_u},i_{k_u+1}} \ldots a_{i_{k_u+v_u}, i_{k_u}}\big) \\
= & a_{ij}
 \left(\sum_{{\sigma}\in \widehat{S_n}} (-1)^{n-u}
 \big( a_{j, i_{k_1+2}},...,  a_{  i_{k_1+v_1},i }  \big)
 \big(a_{i_{k_2},i_{k_2+1}} \ldots a_{i_{k_2+v_2},i_{k_2} }\big) \ldots
 \big( a_{i_{k_u},i_{k_u+1}} \ldots a_{i_{k_u+v_u}, i_{k_u}}\big) \right)\\
 = &a_{ij} R_{ij}.
\end{align*}
Note that there is also a bijection $\Phi$ between $\widehat{S_n}$ and
$S_{n_1}$.   Specifically, $\Phi: \widehat{S_n} \to S_{n_1}$ is defined by $\sigma \mapsto \Phi(\sigma)$, where the left-order cycle representations of $\sigma$ and $\Phi(\sigma)$ are
\begin{align*}
  & \sigma_l =(i,j,i_{k_1+2},..., i_{k_1+v_1} )(i_{k_2},i_{k_2+1},...,i_{k_2+v_2})...
  (i_{k_u},i_{k_u+1},...,i_{k_u+v_u}),\\
  &  \big(\Phi  (  \sigma) \big)_l= (j,i_{k_1+2},..., i_{k_1+v_1})(i_{k_2},i_{k_2+1},...,i_{k_2+v_2})...
  (i_{k_u},i_{k_u+1},...,i_{k_u+v_u}),
\end{align*}
respectively.
 Let $B= (A_{\cdot j}(a_{\cdot i}))^{i,i} =[b_{xy}]\in \mathbb{H} ^{(n-1)\times (n-1)}$, where
 \begin{equation*}
   b_{xy} =
   \begin{cases}
   a_{xy}, &  x\in [n]_i,~ y  \in [n]_i \setminus  \{j\},  \\
 a_{xi},  &    x\in [n]_i,~ y=j.
   \end{cases}
 \end{equation*}
Then,
\begin{align*}
  R_{ij}
  & =
   \sum_{{\sigma}\in \widehat{S_n}} (-1)^{n-u}
  \big( a_{j, i_{k_1+2}},...,  a_{  i_{k_1+v_1},i }  \big)
 \big(a_{i_{k_2},i_{k_2+1}} \ldots a_{i_{k_2+v_2},i_{k_2} }\big) \ldots
 \big( a_{i_{k_u},i_{k_u+1}} \ldots a_{i_{k_u+v_u}, i_{k_u}}\big)  \\
 &=  -1 \cdot \sum_{ \Phi({\sigma}) \in   S_{n_1} } (-1)^{(n-1)-u}
  \big( b_{j, i_{k_1+2}},...,  b_{  i_{k_1+v_1},j }  \big)
 \big(b_{i_{k_2},i_{k_2+1}} \ldots b_{i_{k_2+v_2},i_{k_2} }\big) \ldots
 \big( b_{i_{k_u},i_{k_u+1}} \ldots b_{i_{k_u+v_u}, i_{k_u}}\big)  \\
 &  =  - \sum_{  \Phi (  \sigma) \in {S_{n_1}}}
 {\rm sgn}\big( \Phi (  \sigma)\big)f_{B}\left( \big(\Phi (  \sigma)\big)_l\right)     =  -
 {\rm GDet}_{{\rm P}_{l\text{-}j}} (B).
\end{align*}
Combining \eqref{cofatortheoremitem1eqfirst} and \eqref{cofatortheoremitem1eqsecond}, we conclude that \eqref{cofatortheoremitem1Rijeq} hold.

  \item  The proof of \eqref{cofatortheoremitem2eq} is analogous to that of \eqref{cofatortheoremitem1eq} .

\end{enumerate}
\end{proof}

\begin{definition}\label{PliPljcofactordef}
Let $A=[a_{xy}]\in\mathbb{H}^{n\times n}$, where $x,y\in[n]$.
\begin{enumerate}[(1)]
 \item Let $i \in [n]$, and let ${\rm P}_{l\text{-}i}(S_n, {\rm first}, \prec_{{\triangleleft}})$ be a left-order cycle representation rule.
     We call $R_{ij}$, given by \eqref{cofatortheoremitem1Rijeq}, the right  ${\rm P}_{l\text{-}i}$-cofactor of $a_{ij}$, where $j\in[n]$.

\item Let $j \in [n]$, and let ${\rm P}_{r\text{-}j}(S_n, {\rm last}, \succ_{{\triangleright}})$ be a right-order cycle representation rule. We call $L_{ij}$, given by  \eqref{cofatortheoremitem2Lijeq}, the left ${\rm P}_{r\text{-}j}$-cofactor of $a_{ij}$, where $i\in [n]$.
\end{enumerate}
\end{definition}

\begin{remark}
In particular, under the hypotheses of Definition \ref{PliPljcofactordef}, if ${\rm P}_{l\text{-}i_1}$ is given by Item \eqref{PGdetspecaseremarkitem4} in Remark \ref{PGdetspecaseremark}, then the right ${\rm P}_{l\text{-}i}$-cofactor $R_{i_1,j}$ of $a_{i_1,j}$ reduces to the right $(ij)$-th cofactor defined in \cite[Lemma 2.1]{kyrchei2008cramer}. Similarly, if ${\rm P}_{r\text{-}j_1}$ is given by Item \eqref{PGdetspecaseremarkitem5} in Remark \ref{PGdetspecaseremark}, then the left ${\rm P}_{r\text{-}j}$-cofactor $L_{i,j_1}$ of $a_{i,j_1}$ reduces to the left $(ij)$-th cofactor defined in \cite[Lemma 2.2]{kyrchei2008cramer}.
\end{remark}

Some basic properties of the ${\rm P}$-generalized determinant are presented below.

\begin{theorem}\label{basicproPgdecTh}
Let $A\in\mathbb{H}^{n\times n}$, and let ${\rm P}$ be a cycle representation rule.
\begin{enumerate}[(1)]
  \item\label{basicproPgdecThitem1}  ${\rm GDet}_{{\rm P}}(A_{i \cdot}(0))  ={\rm GDet}_{{\rm P}}(A_{ \cdot j}(0))  =0$.
  \item  If $q \in \mathbb{H}$, then
  \begin{equation*}
     {\rm GDet}_{{\rm P}_{l\text{-}i}}\left(A_{i\cdot}(q\cdot a_{i\cdot})\right)
    = q\cdot {\rm GDet}_{{\rm P}_{l\text{-}i}}(A), ~
      {\rm GDet}_{{\rm P}_{r\text{-}j}}\left(A_{\cdot j}(a_{\cdot j}\cdot q)\right)
    = {\rm GDet}_{{\rm P}_{r\text{-}j}}(A)\cdot q.
  \end{equation*}

  \item\label{basicproPgdecThitem3}
Let  $t\in [n]$  and  $a_{t \cdot }=b+c$, where
    $b=[b_1,\ldots,b_n] \in \mathbb{H}^{1 \times n}$ and $c=[c_1,\ldots,c_n] \in \mathbb{H}^{1 \times n} $.
  Then,
  \[
    {\rm GDet}_{{\rm P}}(A)
    = {\rm GDet}_{{\rm P}}(A_{t\cdot}(b))
    + {\rm GDet}_{{\rm P}}(A_{t\cdot}(c)).
  \]

  \item   Let   $t\in [n]$  and $a_{\cdot t}=b+c$, where $b=[b_1,\ldots,b_n]^T \in \mathbb{H}^{n \times 1}$ and $c=[c_1,\ldots,c_n]^T\in \mathbb{H}^{n \times 1}$.
Then,
  \[
    {\rm GDet}_{{\rm P}}(A)
    = {\rm GDet}_{{\rm P}}(A_{\cdot t}(b))
    + {\rm GDet}_{{\rm P}}(A_{\cdot t}(c)).
  \]
\end{enumerate}
\end{theorem}

\begin{proof}
These properties follow immediately from Definition \ref{Pgdetdeinition}, hence the proofs are omitted.
\end{proof}

\begin{remark}
The properties of row-column determinants given in \cite[Theorems 2.1--2.5]{kyrchei2008cramer} are direct consequences of Theorem~\ref{basicproPgdecTh}.
\end{remark}

 A connection is established between the ${\rm P}_l$-generalized determinant ${\rm GDet}_{{\rm P}_{l}}(A)$ and the ${\rm P}_r$-generalized determinant ${\rm GDet}_{{\rm P}_{r}}(A)$.

\begin{definition}\label{GenassLRRLdecDef}
Let ${\rm P}_l(S_n,\mathrm{first},\prec_{\triangleleft})$ be a left-order cycle representation rule,
let $\triangleright$ be the reverse order of $\triangleleft$
on $[n]$,
and let $\succ_{\triangleright}:= \succ_{\triangleright}\big(   \mathrm{cyc}(\sigma), \triangleright, \mathrm{first}  \big)$
for any $\sigma \in S_n$.
Then,   the rule
${\rm P}_r(S_n,\mathrm{first}, \succ_{\triangleright})$ is called the right-order cycle
representation rule associated with ${\rm P}_l$, and is denoted by $\mathrm{P}_{r'} = (\mathrm{P}_l)'$.
That is, under the rule ${\rm P}_{l}$, if the left-order cycle representation of a permutation $\sigma\in S_n$ is
\begin{equation}\label{genLefoderPlisigmeq}
  \sigma_l=(i_1,i_{k_1+1},\ldots,i_{k_1+v_1})(i_{k_2},i_{k_2+1},\ldots,i_{k_2+v_2})\ldots(i_{k_u},i_{k_u+1},\ldots,i_{k_u+v_u}),
\end{equation}
then the right-order cycle representation of $\sigma$ under the rule ${\rm P}_{r'}$ is
\begin{equation}\label{genRightoderPrisigmeq}
  \sigma_{r}=(i_{k_u+v_u},\ldots,i_{k_u+1},i_{k_u})\ldots(i_{k_2+v_2},\ldots,i_{k_2+1},i_{k_2})(i_{k_1+v_1},\ldots,i_{k_1+1},i_1).
\end{equation}
Similarly,
given another right-order cycle representation rule ${\rm P}_{r}$,
one can define the left-order cycle representation rule associated with ${\rm P}_{r}$,  written as ${\rm P}_{l'}=({\rm P}_{r})'$.

\end{definition}

\begin{remark}
By Definition \ref{GenassLRRLdecDef}, it follows immediately that
\[
(({\rm P}_{l})')' = ({\rm P}_{r'})' = {\rm P}_{l},
\quad
(({\rm P}_{r})')' = ({\rm P}_{l'})' = {\rm P}_{r}.
\]
\end{remark}

\begin{theorem}\label{genGDetlGdetRlinkth}
Let $A\in\mathbb{H}^{n\times n}$ and  $i \in [n]$.  Then
\begin{equation*}
\big(  {\rm GDet}_{{\rm P}_{l}}(A^*) \big)^*=    {\rm GDet}_{{\rm P}_{r'}}(A).
\end{equation*}
\end{theorem}

\begin{proof}
Let $A=[a_{xy}]$ and $B=A^*=[b_{xy}]$, where
$b_{xy}= \overline{a_{yx}}$ for all $x,y\in [n]$.
Let $\sigma_l$ and $\sigma_r$ denote the left-order and right-order cycle
representations of $\sigma\in S_n$ under rules ${\rm P}_l$ and ${\rm P}_{r'}$,
respectively, as given in \eqref{genLefoderPlisigmeq} and
\eqref{genRightoderPrisigmeq}.
Then,
\begin{align*}
   \overline{f_B(\sigma_l)}
   &= \overline{(b_{i_1,i_{k_1+1}}b_{i_{k_1+1},i_{k_1+2}}~\ldots~ b_{i_{k_1+v_1},i_1})~\ldots~
     (b_{i_{k_u},i_{k_u+1}}b_{i_{k_u+1},i_{k_u+2}}~\ldots~ b_{i_{k_u+v_u},i_{k_u}})}\\
   &= \overline{(b_{i_{k_u},i_{k_u+1}}b_{i_{k_u+1},i_{k_u+2}}~\ldots~ b_{i_{k_u+v_u},i_{k_u}})}
      ~\ldots~
      \overline{(b_{i_1,i_{k_1+1}}b_{i_{k_1+1},i_{k_1+2}}~\ldots~ b_{i_{k_1+v_1},i_1})}\\
   &= (\overline{b_{i_{k_u+v_u},i_{k_u}}}~\ldots~
      \overline{b_{i_{k_u+1},i_{k_u+2}}}~\overline{b_{i_{k_u},i_{k_u+1}}})
      ~\ldots~
      (\overline{b_{i_{k_1+v_1},i_1}}~\ldots~
      \overline{b_{i_{k_1+1},i_{k_1+2}}}~\overline{b_{i_1,i_{k_1+1}}})\\
   &= (a_{i_{k_u},i_{k_u+v_u}}~\ldots~ a_{i_{k_u+2},i_{k_u+1}}a_{i_{k_u+1},i_{k_u}})
      ~\ldots~
      (a_{i_1,i_{k_1+v_1}}~\ldots~ a_{i_{k_1+2},i_{k_1+1}}a_{i_{k_1+1},i_1})\\
   &= f_A(\sigma_r).
\end{align*}
Hence,
\begin{align*}
\big({\rm GDet}_{{\rm P}_{l}}(A^*)\big)^*
&= \overline{{\rm GDet}_{{\rm P}_{l}}(A^*)}
= \overline{{\rm GDet}_{{\rm P}_{l}}(B)}
= \overline{\sum_{\sigma\in S_n} {\rm sgn}(\sigma) f_B(\sigma_l)}\\
&= \sum_{\sigma\in S_n} {\rm sgn}(\sigma) \overline{f_B(\sigma_l)}
= \sum_{\sigma\in S_n} {\rm sgn}(\sigma) f_A(\sigma_r)\\
&= {\rm GDet}_{{\rm P}_{r'}}(A),
\end{align*}
which completes the proof.
\end{proof}

\begin{remark}
Under the hypotheses of Theorem \ref{genGDetlGdetRlinkth}, since ${\rm GDet}_{{\rm P}_{l}}(A^*)\in\mathbb{H}$,  we have
\begin{equation}\label{gonetroPlrtraneq}
  \overline{{\rm GDet}_{{\rm P}_{l}}(A^*)}
=\big({\rm GDet}_{{\rm P}_{l}}(A^*)\big)^*
={\rm GDet}_{{\rm P}_{r'}}(A).
\end{equation}
Similarly,
\begin{equation}\label{gonetroPrltraneq}
  \overline{{\rm GDet}_{{\rm P}_{r}}(A^*)}
=\big({\rm GDet}_{{\rm P}_{r}}(A^*)\big)^*
={\rm GDet}_{{\rm P}_{l'}}(A).
\end{equation}
Equations \eqref{gonetroPlrtraneq} and \eqref{gonetroPrltraneq} imply that,
for the ${\rm P}$-generalized determinant, it is sufficient to restrict
attention to the ${\rm P}_l$-generalized determinant (or equivalently, the
${\rm P}_r$-generalized determinant).
In particular, if ${\rm P}_{l}$ is given by Item \eqref{PGdetspecaseremarkitem4} in Remark \ref{PGdetspecaseremark}, then
\[
{\rm GDet}_{{\rm P}_{l}}(A)={\rm rdet}_i(A),
\quad
{\rm GDet}_{{\rm P}_{r'}}(A)={\rm cdet}_i(A),
\]
and we further obtain  \cite[Theorem 2.6]{kyrchei2008cramer}, i.e.,
\[
\overline{{\rm rdet}_i(A^*)}
=\big({\rm rdet}_i(A^*)\big)^*
={\rm cdet}_i(A).
\]
\end{remark}

\section{$\rm P$-generalized determinant of quaternion Hermitian matrices}
\label{PgenrdetHersection}
The Moore,  Xie, Chen, and row-column determinants of quaternion Hermitian matrices all have rich properties.
The study of the ${\rm P}$-generalized determinant for quaternion Hermitian matrices  also constitutes  an important part of this work.
In this section, we first show that the ${\rm P}$-generalized determinant of a Hermitian quaternion matrix is a real number. As a key result of this paper, this shows that for a Hermitian quaternion matrix $A$, its ${\rm P}$-generalized determinant is independent of the choice of the rule ${\rm P}$ and depends only  on $A$.
For convenience, for $d=1,\ldots,u$, let
$\sigma_{k_d}=(i_{k_d},i_{k_d+1},\ldots,i_{k_d+v_d})$
denote the $d$-th cycle of $\sigma_l$ given in \eqref{sigmmaldeceq}, and let
\begin{equation*}
   f_A^c(\sigma_{k_d}) = a_{i_{k_d},i_{k_d+1}} \cdots a_{i_{k_d+v_d}, i_{k_d}},
\end{equation*}
denote the $d$-th factor of $f_A(\sigma_l)$ given in \eqref{fasiggmaldeceq}.

\begin{theorem}\label{HerPGdetmainproth}
Let $A\in\mathbb{H}^{n\times n}$ be Hermitian.
Then ${\rm GDet}_{{\rm P}}(A)$ is a real number independent of the cycle representation rule ${\rm P}$.

\end{theorem}

\begin{proof}
Let $A=[a_{xy}]$ and $A=A^*$, where  $\overline{a_{xy}}=a_{yx}$ for all $x,y\in[n]$. We complete the proof in four  steps.
\begin{enumerate}[(I)]
  \item\label{fistrealnitem} Let ${\rm P}_l={\rm P}_l(S_n, \operatorname{first}, \prec_{\triangleleft})$ be a left-order cycle representation rule. First, we prove that
\begin{equation*}
  {\rm GDet}_{{\rm P}_l}(A)\in\mathbb{R}.
\end{equation*}
Indeed, under the rule ${\rm P}_l$, let the left-order cycle representation of $\sigma\in S_n$ be
\begin{equation}\label{Plsigamlcfeneq}
  \sigma_l = \sigma_{k_1}\sigma_{k_2}\ldots\sigma_{k_u},
\end{equation}
where $  \sigma_{k_d}=(i_{k_d},i_{k_d+1},\ldots,i_{k_d+v_d})$ for $d=1,2,\ldots,u$.
\begin{enumerate}[(i)]
  \item\label{firstitemi} Suppose $v_d\leq 1$ for all $d=1,2,\ldots,u$. Then
  \begin{equation}\label{fAsigkdinReq}
      f_A^c(\sigma_{k_d})=
      \begin{cases}
      a_{k_d,k_d}, & v_d=0,\\
      \vert a_{k_d,k_d+1} \vert^2, & v_d=1,
      \end{cases}
  \end{equation}
  which implies $f_A^c(\sigma_{k_d})\in\mathbb{R}$ and hence $f_A(\sigma_l)\in\mathbb{R}$.

  \item\label{firstitemii} Assume that  there exists at least one index  $v_d$ in \eqref{Plsigamlcfeneq} such that  $v_d\geq2$.
      Let $v_{w_1},\ldots,v_{w_p}\geq2$, where $p\geq1$ and ${w_1},\ldots,{w_p}\in\{1,\ldots,u\}$ satisfy
   $
     i_{k_{w_1}}\triangleleft\ldots\triangleleft i_{k_{w_p}}
$.
  Then,  for $z\in\{1,\ldots,p\}$, we have
  \begin{equation}\label{Fcongongesigameq}
  \begin{aligned}
         \overline{f_A^c(\sigma_{k_{w_z}})}
     &=\overline{a_{i_{k_{w_z}},i_{k_{w_z}+1}}a_{i_{k_{w_z}+1},i_{k_{w_z}+2}}~\ldots ~ a_{i_{k_{w_z}+v_{w_z}},i_{k_{w_z}}}}\\
     &=\overline{a_{i_{k_{w_z}+v_{w_z}},i_{k_{w_z}}}}~\ldots ~
       \overline{a_{i_{k_{w_z}+1},i_{k_{w_z}+2}}}~
       \overline{a_{i_{k_{w_z}},i_{k_{w_z}+1}}}\\
     &=a_{i_{k_{w_z}},i_{k_{w_z}+v_{w_z}}}~\ldots ~
       a_{i_{k_{w_z}+2},i_{k_{w_z}+1}}a_{i_{k_{w_z}+1},i_{k_{w_z}}}\\
     &=f_A^c(\overline{\sigma_{k_{w_z}}}),
  \end{aligned}
  \end{equation}
  where $\overline{\sigma_{k_{w_z}}}=(i_{k_{w_z}},i_{k_{w_z}+v_{w_z}},\ldots,i_{k_{w_z}+1})$.
 Let
 \begin{equation}\label{Tsigmaset2peq}
T_{\sigma}=\left\{ \tau_l= \tau_{k_1}\tau_{k_2}\ldots\tau_{k_u}\;\middle|\;
    \tau_{k_d}=
    \begin{cases}
    \sigma_{k_d}, & d\in\{1,\ldots,u\}\setminus\{{w_1},\ldots,{w_p}\},\\
    \sigma_{k_d}\text{ or }\overline{\sigma_{k_d}}, & d\in\{{w_1},\ldots,{w_p}\},
    \end{cases}
\right\}
 \end{equation}
 be the set consisting of $2^p$ permutations in $S_n$. From Item \eqref{firstitemi}, it follows that
 \begin{equation}\label{alpha12channeq}
   \alpha=\prod_{d\in\{1,\ldots,u\}\setminus\{{w_1},\ldots,{w_p}\}}f_A^c(\tau_{k_d})\in\mathbb{R}.
 \end{equation}
Furthermore, using Lemma \ref{kyrchei2008cramerLemma31}, we obtain
\begin{equation}\label{sigmaRnueq}
  \begin{aligned}
 \sum_{\tau_l\in T_{\sigma}}{\rm sgn}(\tau)f_A(\tau_l)
  &={\rm sgn}(\sigma)\sum_{\tau_l\in T_{\sigma}}f_A^c(\tau_{k_1})\ldots f_A^c(\tau_{k_u})\\
  &={\rm sgn}(\sigma)\alpha \sum_{\tau_l\in T_{\sigma}}f_A^c(\tau_{k_{w_1}})\ldots f_A^c(\tau_{k_{w_p}})\\
  &=(-1)^{n-u}\alpha\,{\rm tr}(f_A^c(\tau_{k_{w_1}}))\ldots{\rm tr}(f_A^c(\tau_{k_{w_p}}))\in\mathbb{R}.
 \end{aligned}
\end{equation}

\end{enumerate}
If $v_d\leq 1$ for all $d=1,2,\ldots,u$, let $T_{\sigma}=\{\sigma_l\}$.
Then, for a permutation $\delta\in S_n$,
 $T_{\delta}=T_{\sigma}$ if and only
if $\delta_l\in T_{\sigma}$, and
$T_{\delta}\cap T_{\sigma}=\emptyset$ otherwise.
Hence,  there exist finitely many distinct permutations $\delta,\eta,\ldots,\gamma\in S_n$
such that \(S_{n}\) is the disjoint union of \(T_{\delta},T_{\eta},\dots,T_{\gamma}\).
Thus, from Items \eqref{firstitemi} and \eqref{firstitemii},  we derive
\begin{equation}\label{GDetPlAsumdaegeq}
  \begin{aligned}
  {\rm GDet}_{{\rm P}_l}(A)
   &=\sum_{\sigma\in S_n}{\rm sgn}(\sigma)f_A(\sigma_l)\\
  &=\sum_{\delta_l\in T_{\delta}}{\rm sgn}(\delta)f_A(\delta_l)
    +\sum_{\eta_l\in T_{\eta}}{\rm sgn}(\eta)f_A(\eta_l)
    +\ldots+\sum_{\gamma_l\in T_{\gamma}}{\rm sgn}(\gamma)f_A(\gamma_l)
    \in\mathbb{R}.
\end{aligned}
\end{equation}
  \item\label{secondPlQleqitem}
  Let ${\rm P}_l={\rm P}_l(S_n,\operatorname{first},\prec_{\triangleleft})$ and ${\rm Q}_l={\rm Q}_l(S_n,\operatorname{first}_Q,\prec_{\triangleleft_Q})$ be two distinct left-order cycle representation rules. We next prove that
   \begin{equation}\label{GDetPleqQleq}
     {\rm GDet}_{{\rm P}_l}(A)={\rm GDet}_{{\rm Q}_l}(A).
   \end{equation}
   Indeed, let $\sigma\in S_n$. Under the rule ${\rm P}_l$, let the left-order cycle representation of $\sigma$ be \eqref{Plsigamlcfeneq}. Under the rule ${\rm Q}_l$,  let the left-order cycle representation of $\sigma$ be
   \begin{equation*}
  \mathring{\sigma}_l=\mathring{\sigma}_{t_1}\mathring{\sigma}_{t_2}\ldots\mathring{\sigma}_{t_u},
   \end{equation*}
where $
  \mathring{\sigma}_{t_e}=(i_{t_e},i_{t_e+1},\ldots,i_{t_e+g_e})
$ for $e=1,2,\ldots,u$.
Note that there is a bijection $\Omega$ from
$\{\sigma_{k_d}: d=1,\ldots,u\}$ to
$\{\mathring{\sigma}_{t_e}: e=1,\ldots,u\}$ given by
\begin{equation}\label{Omegamapeq}
  \Omega(\sigma_{k_d})=\mathring{\sigma}_{t_e},
\end{equation}
where $\sigma_{k_d}$ and $\mathring{\sigma}_{t_e}$ are the same cycle,   possibly with different selected elements under their left-order representations.

\begin{enumerate}[(i)]
  \item\label{seconstepitem1} Suppose that $v_d\leq1$ for all $d=1,2,\ldots,u$. Then $g_e\leq1$ for all $e=1,2,\ldots,u$.
        For any $\sigma_{k_d}$, by Item \eqref{firstitemi} and \eqref{Omegamapeq}, we have
      \begin{equation*}
            f_A^c(\sigma_{k_d})=f_A^c(\Omega(\sigma_{k_d}))=f_A^c(\mathring{\sigma}_{t_e})\in\mathbb{R}.
      \end{equation*}
     Therefore,
       \begin{equation}
      f_A(\sigma_l)=\prod_{d=1}^{u}f_A^c(\sigma_{k_d})
      =\prod_{e=1}^{u}f_A^c(\mathring{\sigma}_{t_e})=f_A(\mathring{\sigma}_l).
      \end{equation}

  \item\label{seconstepitem2} Assume that in \eqref{Plsigamlcfeneq},  there exists at least one index  $v_d$ such that $v_d\geq2$.
Let $v_{w_1},\ldots,v_{w_p}\geq2$ be as given in Item \eqref{firstitemii}.
Using the map $\Omega$ given in \eqref{Omegamapeq}, let
\begin{equation}\label{Omegawzhzeq}
    \Omega({\sigma}_{k_{w_z}})=\mathring{\sigma}_{t_{h_z}},
\end{equation}
where $z=1,\ldots,p$. Then $g_{h_1},\ldots,g_{h_p}\geq2$, where ${h_1},\ldots,{h_p}\in\{1,\ldots,u\}$ and $p\geq1$. Similarly to Item \eqref{firstitemii}, let
 \begin{align*}
T_{\mathring{\sigma}}=\left\{\mathring{\tau}_{t_1}\mathring{\tau}_{t_2}\ldots\mathring{\tau}_{t_u}\;\middle|\;
    \mathring{\tau}_{t_e}=
    \begin{cases}
    \mathring{\sigma}_{t_e}, & e\in\{1,\ldots,u\}\setminus\{{h_1},\ldots,{h_p}\},\\
    \mathring{\sigma}_{t_e}\text{ or }\overline{\mathring{\sigma}_{t_e}}, & e\in\{{h_1},\ldots,{h_p}\},
    \end{cases}
\right\},
 \end{align*}
 where $\overline{\mathring{\sigma}_{t_e}}=(i_{t_e},i_{t_e+g_e},\ldots,i_{t_e+1})$. Let
   \begin{equation*}
  \mathring{\alpha}=\prod_{e\in\{1,\ldots,u\}\setminus\{{h_1},\ldots,{h_p}\}}
  f_A^c(\mathring{\tau}_{t_e})\in\mathbb{R}.
 \end{equation*}
 By Item \eqref{seconstepitem1}, we have $\mathring{\alpha}=\alpha$, where $\alpha$ is given in \eqref{alpha12channeq}. From \eqref{Omegawzhzeq}, there exist nonnegative integers $q_z$ such that
$
  i_{ t_{h_z}}=i_{k_{w_z}+q_z}
$
  for $z=1,\ldots,p$.
Using the trace identity
\begin{equation}\label{quatraceproeq}
  \operatorname{tr}(q_1 q_2)=\operatorname{tr}( q_2q_1)
\end{equation}
 for $q_1, q_2 \in \mathbb{H}$, we derive
 \begin{align*}
 {\rm tr}\left(f_A^c(\tau_{k_{w_z}})\right)
 &={\rm tr}\left(a_{i_{k_{w_z}},i_{k_{w_z}+1}}\ldots a_{i_{k_{w_z}+q_z-1},i_{k_{w_z}+q_z}}
   a_{i_{k_{w_z}+q_z},i_{k_{w_z}+q_z+1}}\ldots a_{i_{k_{w_z}+v_{w_z}},i_{k_{w_z}}}\right)\\
 &={\rm tr}\left(\big(a_{i_{k_{w_z}},i_{k_{w_z}+1}}\ldots a_{i_{t_{h_z}-1},i_{t_{h_z}}}\big)
   \big(a_{i_{t_{h_z}},i_{t_{h_z}+1}}\ldots a_{i_{k_{w_z}+v_{w_z}},i_{k_{w_z}}}\big)\right)\\
 &={\rm tr}\left(\big(a_{i_{t_{h_z}},i_{t_{h_z}+1}}\ldots a_{i_{k_{w_z}+v_{w_z}},i_{k_{w_z}}}\big)
   \big(a_{i_{k_{w_z}},i_{k_{w_z}+1}}\ldots a_{i_{t_{h_z}-1},i_{t_{h_z}}}\big)\right)\\
 &={\rm tr}\left(a_{i_{t_{h_z}},i_{t_{h_z}+1}}\ldots a_{i_{k_{w_z}+v_{w_z}},i_{k_{w_z}}}
   a_{i_{k_{w_z}},i_{k_{w_z}+1}}\ldots a_{i_{t_{h_z}+g_{h_z}},i_{t_{h_z}}}\right)\\
 &={\rm tr}\left(f_A^c(\mathring{\tau}_{t_{h_z}})\right).
 \end{align*}
   Hence, by Lemma \ref{kyrchei2008cramerLemma31}, we obtain
\begin{equation*}
    \begin{aligned}
 \sum_{\mathring{\tau}_l\in T_{\mathring{\sigma}}}{\rm sgn}(\mathring{\tau})f_A(\mathring{\tau}_l)
  &={\rm sgn}({\sigma})\sum_{\mathring{\tau}_l\in T_{\mathring{\sigma}}}f_A^c(\mathring{\tau}_{t_1})\ldots f_A^c(\mathring{\tau}_{t_u})\\
  &=(-1)^{n-u}\mathring{\alpha}\,{\rm tr}(f_A^c(\mathring{\tau}_{t_{h_1}}))\ldots{\rm tr}(f_A^c(\mathring{\tau}_{t_{h_p}}))\\
  &=(-1)^{n-u}\alpha\,{\rm tr}(f_A^c(\tau_{k_{w_1}}))\ldots{\rm tr}(f_A^c(\tau_{k_{w_p}}))\\
  &=\sum_{\tau_l\in T_{\sigma}}{\rm sgn}(\tau)f_A(\tau_l).
  \end{aligned}
\end{equation*}
\end{enumerate}
Using Items \eqref{seconstepitem1}, \eqref{seconstepitem2}, and \eqref{GDetPlAsumdaegeq}, we conclude that \eqref{GDetPleqQleq} holds.

\item\label{thirdburleqitem} Let ${\rm P}_r={\rm P}_r(S_n,\operatorname{last},\succ_{\triangleright})$ be a right-order cycle representation rule. Since  $A=A^*$,
     using \eqref{gonetroPrltraneq} and Item \eqref{fistrealnitem},
    we obtain
\begin{equation*}
{\rm GDet}_{{\rm P}_r}(A)=\overline{{\rm GDet}_{{\rm P}_{l'}}(A^*)}={\rm GDet}_{{\rm P}_{l'}}(A).
\end{equation*}

 \item Finally,  let ${\rm P}$ and ${\rm Q}$ be two distinct cycle representation rules on $S_n$. By Items \eqref{fistrealnitem}--\eqref{thirdburleqitem}, we have
\begin{equation*}
  {\rm GDet}_{{\rm P}}(A)={\rm GDet}_{{\rm Q}}(A)\in\mathbb{R},
\end{equation*}
which completes the proof.
\end{enumerate}

\end{proof}

\begin{remark}\label{GPdeteqrcdeteqchendetremark}
Combining Theorem \ref{HerPGdetmainproth} and Remark \ref{PGdetspecaseremark}, for a Hermitian matrix $A\in\mathbb{H}^{n\times n}$,   we obtain
\begin{align*}
   {\rm GDet}_{{\rm P}}(A)
   &={\rm GDet}_{{\rm P}_{l\text{-}i}}(A)
    ={\rm GDet}_{{\rm P}_{r\text{-}j}}(A)
    ={\rm Mdet}(A)
    ={\rm Cdet}(A)
    ={\rm Xdet}(A)
    ={\rm rdet}_i(A)
    ={\rm cdet}_j(A)\\
   &={\rm DMdet}(A)
    ={\rm CMdet}(A)
    ={\rm DCMdet}(A)
    ={\rm Drdet}_i(A)
    ={\rm CCdet}(A)
    ={\rm Dcdet}_j(A),
\end{align*}
where $i=1,\ldots,n$ and $j=1,\ldots,n$.
In this case,  ${\rm GDet}_{{\rm P}}(A)$ is called the determinant of the Hermitian matrix $A$, denoted by ${\rm Det}(A)$.
Note that ${\rm Det}(A)={\rm det}(A)$ for a Hermitian matrix $A$, where
${\rm det}(A):={\rm rdet}_i(A)={\rm cdet}_j(A)$ for all $i,j=1,\ldots,n$
(see \cite[Remark 3.1]{kyrchei2008cramer}).
\end{remark}

A well-known fact for complex matrices is that if one row (or column) is replaced by another row (or column), then the determinant of the resulting matrix is necessarily zero. However, this conclusion fails for the ${\rm P}$-generalized determinant of
quaternion matrices under an unrestricted rule ${\rm P}$,
even for Hermitian matrices.

 \begin{example}
 let
\[
A=\begin{pmatrix}
1 & \mathbf{i} & \mathbf{j}\\
-\mathbf{i} & 1 & 0\\
-\mathbf{j} & 0 & 1
\end{pmatrix}, \quad
B=A_{2\cdot}(a_{1\cdot})=
\begin{pmatrix}
1 & \mathbf i & \mathbf j\\
1 & \mathbf i & \mathbf j\\
-\mathbf j & 0 & 1
\end{pmatrix},
\]
and let $B=[b_{xy}]$ for $x,y = 1,2,3$. Then,
\[
\begin{aligned}
\operatorname{rdet}_3 (B)
&= b_{33}b_{11}b_{22}
- b_{33}b_{12}b_{21}
- b_{31}b_{13}b_{22}
- b_{32}b_{23}b_{11}
+ b_{31}b_{12}b_{23}
+ b_{32}b_{21}b_{13}
=-2\mathbf{i} \neq 0.
\end{aligned}
\]
\end{example}

The following theorem imposes restrictions on the rule ${\rm P}$ to ensure that the ${\rm P}$-generalized determinant vanishes when a row or column of a quaternion Hermitian matrix is replaced by another.

\begin{theorem}\label{ideteoctworeqth}
Let \(A=A^*\in\mathbb{H}^{n\times n}\) and let \(i,j\in[n]\) with \(i\neq j\). Assume that the cycle representation rule \({\rm P}\) satisfies the following:
\begin{enumerate}[(I)]
  \item If \({\rm P}\) is the left-order cycle rule \({\rm P}_l(S_n,\operatorname{first},\prec_{\triangleleft})\), then:
    \begin{itemize}
      \item if \(i,j\in Q \) for \(Q\in\mathcal{Q}([n])\), then \(\operatorname{first} (Q) \in\{i,j\}\);
      \item if \(i\in Q_1 \) and \(j\in Q_2 \) for \(Q_1,Q_2\in\mathcal{Q}([n])\) with \(Q_1\neq Q_2\), then
          \begin{equation*}
            \operatorname{first} (Q_1)  = i \text{ or }\operatorname{first} (Q_2)  =j.
          \end{equation*}
    \end{itemize}
  \item If \({\rm P}\) is the right-order cycle rule \({\rm P}_r(S_n,\operatorname{last},\succ_{\triangleright})\), then:
    \begin{itemize}
      \item if \(i,j\in Q \) for \(Q\in\mathcal{Q}([n])\), then \(\operatorname{last} (Q) \in\{i,j\}\);

      \item if \(i\in Q_1 \) and \(j\in Q_2 \) for \(Q_1,Q_2\in\mathcal{Q}([n])\) with \(Q_1\neq Q_2\), then
                   \begin{equation*}
            \operatorname{last} (Q_1) = i \text{ or }\operatorname{last} (Q_2) = j.
          \end{equation*}
    \end{itemize}
\end{enumerate}
Then,
\begin{equation}\label{GDetPAjairowcolmaineq}
    {\rm GDet}_{{\rm P}}\left(A_{j\cdot}(a_{i\cdot})\right)
  ={\rm GDet}_{{\rm P}}\left(A_{\cdot j}(a_{\cdot i})\right)
  =0.
\end{equation}
\end{theorem}

\begin{proof}
Let $A=[a_{xy}]$ with $a_{xy} = \overline{a_{yx}}$ for all $x,y \in [n]$.
\begin{enumerate}[{(1)}]
\item\label{ideteoctworeqthAjaiBitem1}
First, we prove that
\begin{equation}\label{GdetAjcaiceqzeroeq}
    {\rm GDet}_{{\rm P}}\left(A_{j\cdot}(a_{i\cdot})\right)
  =0.
\end{equation}
Let $B=A_{j\cdot}(a_{i\cdot})=[b_{xy}]$ for all $x,y \in [n]$,
where
\begin{equation}\label{bjyeqbiyeqaiyeq}
b_{xy}=
\begin{cases}
a_{iy}=b_{iy}, & x=j,\  y \in   [n], \\
a_{xy}=\overline{a_{yx}}=\overline{b_{yx}}, & x\ne j,\ y\ne j,\\
a_{xy}, & x\ne j,\ y=j,
\end{cases}
\end{equation}
by $A=A^*$.

\begin{enumerate}[(I)]
\item\label{ideteoctworeqthitemI}
Let ${\rm P}$ be a left-order cycle rule ${\rm P}_l(S_n,\operatorname{first},\prec_{\triangleleft})$.
Let the left-order cycle representation of a permutation $\sigma \in S_n$ under the rule ${\rm P}_l$ be $\sigma_l$ given in \eqref{Plsigamlcfeneq}.
Let $ \sigma _{k_I}, \sigma _{k_J} \in  \{ \sigma_{k_1},\sigma_{k_2},\ldots,\sigma_{k_u}\}$ satisfy $i \in {\rm supp} (\sigma _{k_I})$ and $j \in {\rm supp} (\sigma _{k_J})$.
We claim  that for $ \sigma_l$, there exist
   $\sigma_{(2)},   \sigma_{(3)},  \sigma_{(4)} \in S_n$
  with left-order cycle representations
   under ${\rm P}_l$ such that
\begin{equation}\label{sigma1234sumeqzeroeq}
  {\rm sgn}(\sigma)  f_{B}(\sigma_l)   +
   {\rm sgn}(\sigma_{(2)})    f_{B}(\sigma_{(2)})
    +{\rm sgn}(\sigma_{(3)})  f_{B}(\sigma_{(3)})
     +{\rm sgn}(\sigma_{(4)}) f_{B}(\sigma_{(4)}) = 0.
\end{equation}
   Without loss of generality,    assume that
   if \(i,j\in Q \) for \(Q\in\mathcal{Q}([n])\), then \(\operatorname{first} (Q) = j\).
  \begin{description}
    \item[Case 1.1:] Let $\sigma_{k_I}\neq\sigma_{k_J}$. Without loss of generality, assume $J<I$.
\begin{description}
  \item[Case 1.1.1:]\label{ideteoctworeqthcase11}
    Let
  $ \operatorname{first} ({\rm supp} (\sigma _{k_I} ))  =i  $ and
    $ \operatorname{first}  ({\rm supp} (\sigma _{k_J} ))   =j  $, i.e.,
    \begin{equation*}
      i_{k_I} =i  \quad \text{and} \quad  i_{k_J}=j.
    \end{equation*}
 Then
\begin{align*}
    \sigma_l&= \sigma_{k_1}\ldots\sigma_{k_J}\ldots\sigma_{k_I} \ldots \sigma_{k_u},
\end{align*}
where
\begin{align}
\sigma_{k_J}&=    (i_{k_J},\ldots,i_{k_J+v_J})=   (j,i_{k_J+1},\ldots,i_{k_J+v_J}) , \label{case01skJsKIeq}\\
\sigma_{k_I}& =(i_{k_I},\ldots,i_{k_I+v_I})  =(i,i_{k_I+1},\ldots,i_{k_I+v_I}).  \nonumber
\end{align}
    Let
   \begin{equation}\label{mathcalF123eq}
      \mathcal{F}_1 =  \prod \limits_{1\leq d <J}  f_{B}^c(\sigma_{k_d}) ,~
      \mathcal{F}_2 = \prod \limits_{ J< d <I}  f_{B}^c(\sigma_{k_d}),~
      \mathcal{F}_3 = \prod \limits_{ I< d \leq u}  f_{B}^c(\sigma_{k_d}),
    \end{equation}
where $  \mathcal{F}_1=1$ if $J=1$, $  \mathcal{F}_2=1$ if $I=J+1$,
and $  \mathcal{F}_3=1$ if $I=u$.
Then
\begin{equation}\label{Gdepeqz0case01f1}
 f_{B}(\sigma_l)=   \mathcal{F}_1  f_{B}^c(\sigma_{k_{J}})  \mathcal{F}_2   f_{B}^c(\sigma_{k_{I}})    \mathcal{F}_3.
\end{equation}
Let
\begin{equation*}
  \sigma_{(2)}= \sigma_{k_1}\ldots\sigma_{k_J}\ldots \overline{\sigma_{k_I}} \ldots \sigma_{k_u} \in S_n,
\end{equation*}
where $\overline{\sigma_{k_{I}}}=(i_{I},i_{k_{I}+v_{I}},\ldots,i_{k_{I}+1})$. In analogy with \eqref{Fcongongesigameq} and using \eqref{bjyeqbiyeqaiyeq}, we have
\begin{equation}\label{gongesigKieqfBeq}
  \overline{f_{B}^c (\sigma_{k_{I}}) } = f_{B}^c (\overline{\sigma_{k_{I}}}).
\end{equation}
It follows that $\sigma_{(2)}$ is some permutation in $S_n$  with the left-order cycle representation under ${\rm P}_l$.
Then
\begin{equation}\label{Gdepeqz0case01f2}
  \begin{aligned}
   f_{B}(\sigma_{(2)})= &  \mathcal{F}_1     f_{B}^c({\sigma_{k_J}} )  \mathcal{F}_2   f_{B}^c( \overline{{\sigma_{k_I}}} )   \mathcal{F}_3
 =
\mathcal{F}_1  f_{B}^c(\sigma_{k_{J}})\mathcal{F}_2 \overline{f_{B}^c(\sigma_{k_{I}})}\mathcal{F}_3.
\end{aligned}
\end{equation}
 Let
\begin{equation*}
  \sigma_{(3)}=  \sigma_{k_1}\ldots  \widetilde{\sigma_{k_J}}
 \ldots   \sigma_{k_{(I-1)}}    \sigma_{k_{(I+1)}} \ldots \sigma_{k_u},
\end{equation*}
where
\begin{equation*}
  \widetilde{\sigma_{k_J}}=(j, i_{k_{I}+1},...,i_{k_{I}+v_{I}},i, i_{k_{J}+1},...,i_{k_{J}+v_{J}}).
\end{equation*}
It follows that $\sigma_{(3)}$ is some permutation in $S_n$  with the left-order cycle representation under ${\rm P}_l$.
Using \eqref{bjyeqbiyeqaiyeq}, we have
\begin{equation}\label{Gdepeqz0case01f3}
\begin{aligned}
   f_{B}(\sigma_{(3)})= &  \mathcal{F}_1     f_{B}^c(\widetilde{\sigma_{k_J}} )  \mathcal{F}_2 \mathcal{F}_3
   \\=& \mathcal{F}_1 (b _{j,i_{k_{I}+1}}~\ldots ~ b_{i_{k_{I}+v_{I}},i} b_{i,i_{k_{J}+1}}~\ldots ~ b_{i_{k_{J}+v_{J}},j}) \mathcal{F}_2 \mathcal{F}_3
   \\=& \mathcal{F}_1(b _{i,i_{k_{I}+1}}~\ldots ~ b_{i_{k_{I}+v_{I}},i} ) ( b_{j,i_{k_{J}+1}}~\ldots ~ b_{i_{k_{J}+v_{J}},j})\mathcal{F}_2 \mathcal{F}_3
     \\=&  \mathcal{F}_1(b_{i_{k_{I}},i_{k_{I}+1}}~\ldots ~ b_{i_{k_{I}+v_{I}},i_{k_{I}}} ) (b_{i_{k_{J}},i_{k_{J}+1}}~\ldots ~ b_{i_{k_{J}+v_{J}},i_{k_{J}}})\mathcal{F}_2 \mathcal{F}_3
       \\=&   \mathcal{F}_1    f_{B}^c(\sigma_{k_{I}})    f_{B}^c(\sigma_{k_{J}})\mathcal{F}_2 \mathcal{F}_3.
\end{aligned}
\end{equation}
 Let
\begin{equation*}
  \sigma_{(4)}=  \sigma_{k_1}\ldots  \widehat{\sigma_{k_J}}
 \ldots   \sigma_{k_{(I-1)}}    \sigma_{k_{(I+1)}} \ldots \sigma_{k_u},
\end{equation*}
where
\begin{equation*}
  \widehat{\sigma_{k_J}}=(j, i_{k_{I}+v_{I}},...,i_{k_{I}+1},i, i_{k_{J}+1},...,i_{k_{J}+v_{J}}).
\end{equation*}
It follows that $\sigma_{(4)}$ is some permutation in $S_n$  with the left-order cycle representation under ${\rm P}_l$.
Using \eqref{bjyeqbiyeqaiyeq} and \eqref{gongesigKieqfBeq}, we have
\begin{equation}\label{Gdepeqz0case01f4}
\begin{aligned}
   f_{B}(\sigma_{(4)})= &  \mathcal{F}_1     f_{B}^c(\widehat{\sigma_{k_J}} )  \mathcal{F}_2 \mathcal{F}_3
     \\=&
   \mathcal{F}_1 (b _{j,i_{k_{I}+v_{I}}}~\ldots ~ b_{i_{k_{I}+1},i} b_{i,i_{k_{J}+1}}~\ldots ~ b_{i_{k_{J}+v_{J}},j}) \mathcal{F}_2 \mathcal{F}_3
   \\=& \mathcal{F}_1(b _{i,i_{k_{I}+v_{I}}}~\ldots ~ b_{i_{k_{I}+1},i} ) ( b_{j,i_{k_{J}+1}}~\ldots ~ b_{i_{k_{J}+v_{J}},j})\mathcal{F}_2 \mathcal{F}_3
     \\=&  \mathcal{F}_1(b_{i_{k_{I}},i_{k_{I}+v_{I}}}~\ldots ~ b_{i_{k_{I}+1},i_{k_{I}}} ) (b_{i_{k_{J}},i_{k_{J}+1}}~\ldots ~ b_{i_{k_{J}+v_{J}},i_{k_{J}}})\mathcal{F}_2 \mathcal{F}_3
       \\=&   \mathcal{F}_1    f_{B}^c( \overline{\sigma_{k_{I}}})    f_{B}^c(\sigma_{k_{J}})\mathcal{F}_2 \mathcal{F}_3
        \\=&   \mathcal{F}_1     \overline{f_{B}^c({\sigma_{k_{I}}}) } f_{B}^c(\sigma_{k_{J}})\mathcal{F}_2 \mathcal{F}_3.
\end{aligned}
\end{equation}
Using \eqref{Gdepeqz0case01f1}, \eqref{Gdepeqz0case01f2}, \eqref{Gdepeqz0case01f3}, and \eqref{Gdepeqz0case01f4}, we obtain
\begin{equation*}
\begin{aligned}
  &{\rm sgn}(\sigma)  f_{B}(\sigma_l)   +
   {\rm sgn}(\sigma_{(2)})    f_{B}(\sigma_{(2)})
    +{\rm sgn}(\sigma_{(3)})  f_{B}(\sigma_{(3)})
     +{\rm sgn}(\sigma_{(4)}) f_{B}(\sigma_{(4)})
  \\=    & (-1)^{n-u}   \mathcal{F}_1  f_{B}^c(\sigma_{k_{J}})  \mathcal{F}_2   f_{B}^c(\sigma_{k_{I}})    \mathcal{F}_3 +
(-1)^{n-u}  \mathcal{F}_1  f_{B}^c(\sigma_{k_{J}})\mathcal{F}_2 \overline{f_{B}^c(\sigma_{k_{I}})}\mathcal{F}_3
\\& +(-1)^{n-(u-1)}  \mathcal{F}_1    f_{B}^c(\sigma_{k_{I}})    f_{B}^c(\sigma_{k_{J}})\mathcal{F}_2 \mathcal{F}_3
+(-1)^{n-(u-1)} \mathcal{F}_1    f_{B}^c( \overline{\sigma_{k_{I}}})    f_{B}^c(\sigma_{k_{J}})\mathcal{F}_2 \mathcal{F}_3
\\=& (-1)^{n-u}   \mathcal{F}_1  f_{B}^c(\sigma_{k_{J}})  \mathcal{F}_2   {\rm tr}(f_{B}^c(\sigma_{k_{I}}) )   \mathcal{F}_3 +  (-1)^{n-(u-1)}  \mathcal{F}_1  {\rm tr}(  f_{B}^c(\sigma_{k_{I}})  )  f_{B}^c(\sigma_{k_{J}})\mathcal{F}_2 \mathcal{F}_3
\\=&   (-1)^{n-u} {\rm tr}(  f_{B}^c(\sigma_{k_{I}})  ) \big(\mathcal{F}_1  f_{B}^c(\sigma_{k_{J}})  \mathcal{F}_2\mathcal{F}_3 - \mathcal{F}_1f_{B}^c(\sigma_{k_{J}})\mathcal{F}_2 \mathcal{F}_3\big)
\\=& 0,
\end{aligned}
\end{equation*}
i.e.,  \eqref{sigma1234sumeqzeroeq} holds.
  \item[Case 1.1.2:]\label{ideteoctworeqthcase12}
 Let
  $ \operatorname{first} ({\rm supp} (\sigma _{k_I} )) = h  \neq i $ and
    $ \operatorname{first}  ({\rm supp} (\sigma _{k_J} ))   =j$, i.e.,
    \begin{equation*}
      i_{k_I}  = h \neq i  \quad \text{and} \quad  i_{k_J}=j.
    \end{equation*}
 Let $v_w \in\{1,...,v_I\}$ be such that $i_{ k_I + v_w } =i$. Then
    \begin{equation*}
     \sigma_l=  \sigma_{k_1}\ldots\sigma_{k_J}\ldots\sigma_{k_I}^{h} \ldots \sigma_{k_u},
    \end{equation*}
    where $\sigma_{k_J}$ is the same as in \eqref{case01skJsKIeq}, and
    \begin{align*}
\sigma_{k_I}^{h}=
 ( h,   i_{k_I +1},\ldots, i_{ k_I + v_w-1},  i , i_{ k_I + v_w+1}   \ldots,  i_{k_I+v_I}).
    \end{align*}
    Then
    \begin{equation}\label{Gdepeqz0case02f1}
       f_{B}(\sigma_l)=   \mathcal{F}_1  f_{B}^c(\sigma_{k_{J}})  \mathcal{F}_2   f_{B}^c(\sigma_{k_{I}^h})    \mathcal{F}_3,
    \end{equation}
    where $\mathcal{F}_1,\mathcal{F}_2,\mathcal{F}_3$ are given in \eqref{mathcalF123eq}.
Let
\begin{equation*}
  \sigma_{(2)}= \sigma_{k_1}\ldots\sigma_{k_J}\ldots \overline{\sigma_{k_I}^h} \ldots \sigma_{k_u},
\end{equation*}
where $\overline{\sigma_{k_{I}}^h} =  (h,  i_{k_I+v_I}, \ldots
 i_{ k_I + v_w+1} ,i ,i_{ k_I + v_w-1}, \ldots i_{k_I +1})$ satisfies
\begin{equation*}
  \overline{f_{B} ^c( \sigma_{k_{I}}^h ) } = {f_{B} ^c (\overline{\sigma_{k_{I}}^h})},
\end{equation*}
 as in \eqref{gongesigKieqfBeq}.
 It follows that $\sigma_{(2)}$ is some permutation in $S_n$  with the left-order cycle representation under ${\rm P}_l$.
  Then
\begin{equation}\label{Gdepeqz0case02f2}
  \begin{aligned}
   f_{B}(\sigma_{(2)})= &  \mathcal{F}_1     f_{B}^c({\sigma_{k_J}} )  \mathcal{F}_2   f_{B}^c( \overline{{\sigma_{k_I}^h}} )   \mathcal{F}_3=
\mathcal{F}_1  f_{B}^c(\sigma_{k_{J}})\mathcal{F}_2 \overline{f_{B}^c(\sigma_{k_{I}^h})}\mathcal{F}_3.
\end{aligned}
\end{equation}
Let
\begin{equation*}
  \sigma_{(3)}=   \sigma_{k_1}\ldots  \widetilde{\sigma_{k_J}}
 \ldots   \sigma_{k_{(I-1)}}    \sigma_{k_{(I+1)}} \ldots \sigma_{k_u},
\end{equation*}
where
\begin{equation*}
   \widetilde{\sigma_{k_J}}=  (j,  i_{k_{I}+v_{w}+1},..., i_{k_{I}+v_{I}},h , i_{k_{I}+1},...,i_{k_{I}+v_{w}-1},i,  i_{k_{J}+1},..., i_{k_{J}+v_{J}}) .
\end{equation*}
It follows that $\sigma_{(3)}$ is some permutation in $S_n$  with the left-order cycle representation under ${\rm P}_l$.
 Using \eqref{bjyeqbiyeqaiyeq}, we have
\begin{equation}\label{Gdepeqz0case02f3}
\begin{aligned}
  & f_{B}(\sigma_{(3)})=   \mathcal{F}_1     f_{B}^c(\widetilde{\sigma_{k_J}} )  \mathcal{F}_2 \mathcal{F}_3
   \\=& \mathcal{F}_1 (b _{j, i_{k_{I}+v_{w}+1}}~\ldots ~ b_{i_{k_{I}+v_{I}},h} b_{h , i_{k_{I}+1}}~\ldots ~
   b_{ i_{k_{I}+v_w-1},i}  b_{ i,  i_{k_{J}+1}},..., b_{ i_{k_{J}+v_J}, j} ) \mathcal{F}_2 \mathcal{F}_3
    \\=& \mathcal{F}_1 \big(b _{i, i_{k_{I}+v_{w}+1}}~\ldots ~ b_{i_{k_{I}+v_{I}},h} b_{h , i_{k_{I}+1}}~\ldots ~
   b_{ i_{k_{I}+v_w-1},i} \big) \big( b_{ j,  i_{k_{J}+1}},..., b_{ i_{k_{J}+v_J}, j} \big) \mathcal{F}_2 \mathcal{F}_3
  \\=&   \mathcal{F}_1 \mathcal{H}_1   f_{B}^c({\sigma_{k_J}} ) \mathcal{F}_2 \mathcal{F}_3,
\end{aligned}
\end{equation}
where $\mathcal{H}_1= b _{i, i_{k_{I}+v_{w}+1}}~\ldots ~ b_{i_{k_{I}+v_{I}},h} b_{h , i_{k_{I}+1}}~\ldots ~
   b_{ i_{k_{I}+v_w-1},i} $.
 Let
\begin{equation*}
  \sigma_{(4)}=  \sigma_{k_1}\ldots  \widehat{\sigma_{k_J}}
 \ldots   \sigma_{k_{(I-1)}}    \sigma_{k_{(I+1)}} \ldots \sigma_{k_u},
\end{equation*}
where
\begin{equation*}
  \widehat{\sigma_{k_J}}= (j,  i_{k_{I}+v_{w}-1},...,    i_{k_{I}+1}, h, i_{k_{I}+v_{I}},..., i_{k_{I}+v_{w}+1},
  i,  i_{k_{J}+1},..., i_{k_{J}+v_{J}}) .
\end{equation*}
It follows that $\sigma_{(4)}$ is some permutation in $S_n$  with the left-order cycle representation under ${\rm P}_l$.
Since,  by \eqref{bjyeqbiyeqaiyeq},
\begin{align*}
\overline{\mathcal{H}_1}&= \overline{b _{i, i_{k_{I}+v_{w}+1}}~\ldots ~ b_{i_{k_{I}+v_{I}},h} b_{h , i_{k_{I}+1}}~\ldots ~
   b_{ i_{k_{I}+v_w-1},i} }
    \\ &=\overline{   b_{ i_{k_{I}+v_w-1},i}}  ~ \ldots ~
    \overline{ b_{h , i_{k_{I}+1}}}   ~
    \overline{ b_{i_{k_{I}+v_{I}},h}}
     ~  \ldots ~
     \overline{ b _{i, i_{k_{I}+v_{w}+1}}}
   \\ &=b _{i,i_{k_{I}+v_{w}-1}}\ldots  b_{i_{k_{I}+1},h} b_{h,i_{k_{I}+v_I}}\ldots  b_{i_{k_{I}+v_w+1, i}},
\end{align*}
it follows that
\begin{equation}\label{Gdepeqz0case02f4}
\begin{aligned}
 &    f_{B}(\sigma_{(4)})= \mathcal{F}_1     f_{B}^c(\widehat{\sigma_{k_J}} )  \mathcal{F}_2 \mathcal{F}_3
     \\=&
   \mathcal{F}_1 \big(b _{j,i_{k_{I}+v_{w}-1}}\ldots  b_{i_{k_{I}+1},h} b_{h,i_{k_{I}+v_I}}\ldots  b_{i_{k_{I}+v_w+1, i}} b_{i, i_{k_J +1 }} ...b_{ i_{k_J +v_J }, j} \big) \mathcal{F}_2 \mathcal{F}_3
   \\=&
   \mathcal{F}_1 \big(b _{i,i_{k_{I}+v_{w}-1}}\ldots  b_{i_{k_{I}+1},h} b_{h,i_{k_{I}+v_I}}\ldots  b_{i_{k_{I}+v_w+1, i}} \big)\big(b_{j, i_{k_J +1 }} ...b_{ i_{k_J +v_J }, j} \big) \mathcal{F}_2 \mathcal{F}_3
      \\=&   \mathcal{F}_1\overline{\mathcal{H}_1 } f_{B}^c({\sigma_{k_J}} )  \mathcal{F}_2 \mathcal{F}_3.
\end{aligned}
\end{equation}
Using \eqref{quatraceproeq}, we have
\begin{equation}\label{trH1eqtrfbhseq}
\begin{aligned}
  {\rm tr}(\mathcal{H}_1)
  &={\rm tr}\big(  ( b _{i, i_{k_{I}+v_{w}+1}}~\ldots ~ b_{i_{k_{I}+v_{I}},h})( b_{h , i_{k_{I}+1}}~\ldots ~
   b_{ i_{k_{I}+v_w-1},i} )\big)
  \\&   = {\rm tr}\big(  ( b _{i_{k_I+v_{w}}, i_{k_{I}+v_{w}+1}}~\ldots ~ b_{i_{k_{I}+v_{I}},h})( b_{h , i_{k_{I}+1}}~\ldots ~
   b_{ i_{k_{I}+v_w-1},i_{k_I+v_{w}}} )\big)
\\&   =  {\rm tr}\big(  ( b_{h , i_{k_{I}+1}}~\ldots ~
   b_{ i_{k_{I}+v_w-1},i_{k_I+v_{w}}} )( b _{i_{k_I+v_{w}}, i_{k_{I}+v_{w}+1}}~\ldots ~ b_{i_{k_{I}+v_{I}},h})\big)
 \\&   =  {\rm tr}\left(f_{B}(\sigma_{k_{I}^h}) \right).
\end{aligned}
\end{equation}
Combining \eqref{Gdepeqz0case02f1}, \eqref{Gdepeqz0case02f2}, \eqref{Gdepeqz0case02f3}, \eqref{Gdepeqz0case02f4}, and \eqref{trH1eqtrfbhseq}, we obtain
\begin{equation*}
\begin{aligned}
  &{\rm sgn}(\sigma)  f_{B}(\sigma_l)   +
   {\rm sgn}(\sigma_{(2)})    f_{B}(\sigma_{(2)})
    +{\rm sgn}(\sigma_{(3)})  f_{B}(\sigma_{(3)})
     +{\rm sgn}(\sigma_{(4)}) f_{B}(\sigma_{(4)})
  \\=    & (-1)^{n-u}   \mathcal{F}_1  f_{B}^c(\sigma_{k_{J}})  \mathcal{F}_2   f_{B}^c(\sigma_{k_{I}^h})    \mathcal{F}_3 +
(-1)^{n-u}  \mathcal{F}_1  f_{B}^c(\sigma_{k_{J}})\mathcal{F}_2 \overline{f_{B}^c(\sigma_{k_{I}^h})}\mathcal{F}_3
\\& +(-1)^{n-(u-1)}  \mathcal{F}_1   \mathcal{H}_1 f_{B}^c(\sigma_{k_{J}})\mathcal{F}_2 \mathcal{F}_3
+(-1)^{n-(u-1)} \mathcal{F}_1   \overline{ \mathcal{H}_1 }    f_{B}^c(\sigma_{k_{J}})\mathcal{F}_2 \mathcal{F}_3
  \\=   & (-1)^{n-u}  {\rm tr}\big(f_{B}^c(\sigma_{k_{I}^h}) \big) \mathcal{F}_1  f_{B}^c(\sigma_{k_{J}})  \mathcal{F}_2   \mathcal{F}_3  +
    (-1)^{n-(u-1)}  {\rm tr}\big( \mathcal{H}_1  \big)
    \mathcal{F}_1  f_{B}^c(\sigma_{k_{J}})  \mathcal{F}_2   \mathcal{F}_3
      \\=   &(-1)^{n-u} \left( {\rm tr}\big(f_{B}^c(\sigma_{k_{I}^h}) ) \mathcal{F}_1  f_{B}^c(\sigma_{k_{J}})  \mathcal{F}_2   \mathcal{F}_3  -{\rm tr}( \mathcal{H}_1  )
    \mathcal{F}_1  f_{B}^c(\sigma_{k_{J}})  \mathcal{F}_2   \mathcal{F}_3  \right)    \\=   & 0,
\end{aligned}
\end{equation*}
i.e.,  \eqref{sigma1234sumeqzeroeq} holds.
 \item[Case 1.1.3:]\label{ideteoctworeqthcase13} Let
  $ \operatorname{first} ({\rm supp} (\sigma _{k_I} ))=i $ and
    $ \operatorname{first}  ({\rm supp} (\sigma _{k_J} ))  = h  \neq j$,  i.e.,
    \begin{equation*}
      i_{k_I}  =i   \quad \text{and} \quad i_{k_J}= h \neq j.
    \end{equation*}
As in Case 1.1.2, there exist
   $\sigma_{(2)},   \sigma_{(3)},  \sigma_{(4)} \in S_n$    with left-order cycle representations
   under ${\rm P}_l$
   such that \eqref{sigma1234sumeqzeroeq} holds.
\end{description}

    \item[Case 1.2:] Let $\sigma _{k_I} = \sigma _{k_J}$, i.e., $k_I = {k_J}$.
        Then $i,j\in{\rm supp}(\sigma_{k_J})$, and by assumption,
$\operatorname{first}({\rm supp}(\sigma_{k_J}))=j$, i.e., $i_{k_J}=j$.
        Let $v_w \in \{1,..., v_{J}\}$ satisfy  $i_{ k_J + v_w } =i $. Then
\begin{align*}
    \sigma_l&= \sigma_{k_1}\ldots\sigma_{k_J}\ldots  \sigma_{k_u}.
\end{align*}
where
\begin{align*}
\sigma_{k_J}
 =   (j,i_{k_J+1},\ldots, i_{ k_J + v_w-1 },i, i_{ k_J + v_w+1 }, \ldots, i_{k_J+v_J}) .
\end{align*}
Note that
\begin{align*}
f_{B}^c(\sigma_{k_J}) &= b_{j , i_{k_J+1}} ... b_{ i_{ k_J + v_w-1 },i } b_{ i, i_{ k_J + v_w+1 } } ...  b_{ i_{k_J+v_J}, j} \\
& = \big(b_{i , i_{k_J+1}} ... b_{ i_{ k_J + v_w-1 },i } \big)\big(b_{ j, i_{ k_J + v_w+1 } } ...  b_{ i_{k_J+v_J}, j}\big) = \mathcal{H}_2 \mathcal{H}_3,
\end{align*}
where
$
   \mathcal{H}_2 = b_{i , i_{k_J+1}} ... b_{ i_{ k_J + v_w-1 },i } $ and
$
  \mathcal{H}_3 = b_{ j, i_{ k_J + v_w+1 } } ...  b_{ i_{k_J+v_J}, j}.
$
\begin{description}
  \item[Case 1.2.1] Assume that the selector $\operatorname{first} $ satisfies: if \(i\in Q_1 \) and \(j\in Q_2 \) for \(Q_1,Q_2\in\mathcal{Q}([n])\) with \(Q_1\neq Q_2\), then $ \operatorname{first} (Q_1)  = i $ and $\operatorname{first} (Q_2)  =j$. In this case, regard $\sigma_l$ as $\sigma_{(3)}$ in Case 1.1.1; similarly, there exist three other permutations in $S_n$ such that \eqref{sigma1234sumeqzeroeq} holds.

  \item[Case 1.2.2] Assume that the selector $\operatorname{first} $ satisfies: if \(i\in Q_1 \) and \(j\in Q_2 \) for \(Q_1,Q_2\in\mathcal{Q}([n])\) with \(Q_1\neq Q_2\), then $ \operatorname{first} (Q_1)  \neq i $ and $\operatorname{first} (Q_2)  =j$.
              In this case, regard $\sigma_l$ as $\sigma_{(3)}$ in Case 1.1.2; similarly, there exist three other permutations in $S_n$ such that \eqref{sigma1234sumeqzeroeq} holds.

  \item[Case 1.2.3] Assume that the selector $\operatorname{first} $ satisfies: if \(i\in Q_1 \) and \(j\in Q_2 \) for \(Q_1,Q_2\in\mathcal{Q}([n])\) with \(Q_1\neq Q_2\), then $ \operatorname{first} (Q_1)  = i $ and $\operatorname{first} (Q_2)  \neq j$.
              In this case, regard $\sigma_l$ as $\sigma_{(3)}$ in Case 1.1.3; similarly, there exist three other permutations in $S_n$ such that \eqref{sigma1234sumeqzeroeq} holds.
\end{description}
\end{description}
Let
\begin{equation*}
  \mathcal{S}_1 = \{ \sigma \in S_n~ | ~\sigma _{k_I} \neq  \sigma _{k_J} \} \text{  and }\mathcal{S}_2 = \{ \sigma \in S_n ~|~ \sigma _{k_I} =   \sigma _{k_J} \}.
\end{equation*}
Then $   \mathcal{S}_1   \cup \mathcal{S}_2 =S_n $ and
$ \mathcal{S}_1   \cap \mathcal{S}_2  = \emptyset  $.
Let a map
\begin{equation*}
  \Upsilon:~ \mathcal{S}_1 \rightarrow \mathcal{S}_2  \quad \text{by}  \quad  \sigma \mapsto  \Upsilon(\sigma)=  \sigma_{(3)},
\end{equation*}
where $ \sigma_{(3)}$ is given in Case 1.1.
From the discussion in Case 1.1,
the definition of $\Upsilon$ is well defined.
From the discussion in  Case 1.2,
$\Upsilon$ is surjective.
Thus, $\Upsilon$ is bijective since both $\mathcal{S}_1$ and $\mathcal{S}_2$ have $\frac{n!}{2}$ elements, where $n!$ denotes the factorial of $n$.
Note that
\begin{equation*}
  \sigma_{(2)} \in \mathcal{S}_1
 \quad \text{and} \quad
  \Upsilon(\sigma_{(2)}) =\sigma_{(4)}  \in \mathcal{S}_2,
\end{equation*}
where  $ \sigma_{(2)}$  and  $ \sigma_{(4)}$ are given in Case 1.1.
Using \eqref{sigma1234sumeqzeroeq}, we obtain
\begin{align*}
    \sum\limits_{\sigma\in {\mathcal{S}}_{1}} {\rm sgn}(\sigma)  f_B(\sigma_l) & =
 -   \sum\limits_{\sigma\in {\mathcal{S}}_{1}} {\rm sgn}(\Upsilon(\sigma))  f_B(\Upsilon(\sigma))
=
  -   \sum\limits_{\sigma\in {\mathcal{S}}_{1}} {\rm sgn}( \sigma_{(3)})  f_B( \sigma_{(3)})\\
&
  =-   \sum\limits_{\sigma \in {\mathcal{S}}_{2}} {\rm sgn}(\sigma )  f_B(\sigma_l).
\end{align*}
Hence,
\begin{align*}
 {\rm GDet}_{{\rm P}_l}\left( B \right) =
 \sum\limits_{\sigma\in S_n} {\rm sgn}(\sigma)  f_B(\sigma_l)    =
  \sum\limits_{\sigma\in {\mathcal{S}}_{1}} {\rm sgn}(\sigma)  f_B(\sigma_l)
  +  \sum\limits_{\sigma\in {\mathcal{S}}_{2}} {\rm sgn}(\sigma)  f_B(\sigma_l)  =0.
\end{align*}

\item\label{ideteoctworeqthitemII}  Let ${\rm P}$ be a right-order cycle rule ${\rm P}_r(S_n,\operatorname{last},\succ_{\triangleright})$.
  The proof of
      \begin{equation*}
    {\rm GDet}_{{\rm P}_r}\left(B\right)
  =0
\end{equation*}
is similar to that of Item \eqref{ideteoctworeqthitemI}, and hence is omitted.
\end{enumerate}

Therefore, \eqref{GdetAjcaiceqzeroeq} follows from Items
\eqref{ideteoctworeqthitemI} and \eqref{ideteoctworeqthitemII}.

  \item\label{ideteoctworeqthAjaiBitem2}
 We  next prove
\begin{equation*}
    {\rm GDet}_{{\rm P}}\left(A_{\cdot j}(a_{\cdot i}) \right)=0.
\end{equation*}
  Indeed,  using Theorem \ref{genGDetlGdetRlinkth} and  \eqref{GdetAjcaiceqzeroeq},  it follows that
  \begin{enumerate}[(I)]
    \item if ${\rm P}$ is a left-order cycle rule ${\rm P}_l(S_n,\operatorname{first},\prec_{\triangleleft})$, then
        \begin{equation*}
          {\rm GDet}_{{\rm P}_l}\left(A_{\cdot j}(a_{\cdot i})\right) =
          \big({\rm GDet}_{{\rm P}_{r'}}((A_{\cdot j}(a_{\cdot i}))^*)\big)^*
          = \big({\rm GDet}_{{\rm P}_{r'}}(A_{ j\cdot }(a_{ i\cdot}))\big)^*=0;
        \end{equation*}

    \item if ${\rm P}$ is a right-order cycle rule ${\rm P}_r(S_n,\operatorname{last},\succ_{\triangleright})$, then
            \begin{equation*}
          {\rm GDet}_{{\rm P}_r}\left(A_{\cdot j}(a_{\cdot i})\right) =
          \big({\rm GDet}_{{\rm P}_{l'}}((A_{\cdot j}(a_{\cdot i}))^*)\big)^*
          = \big({\rm GDet}_{{\rm P}_{l'}}(A_{ j\cdot }(a_{ i\cdot}))\big)^*=0.
        \end{equation*}
  \end{enumerate}

\end{enumerate}
By Items \eqref {ideteoctworeqthAjaiBitem1} and
\eqref{ideteoctworeqthAjaiBitem2}, we obtain that
\eqref{GDetPAjairowcolmaineq} holds.
\end{proof}

\begin{theorem}\label{beilianghleqth}
Under the hypotheses of Theorem \ref{ideteoctworeqth}, let $b \in \mathbb{H}$.  Then,
\begin{equation}\label{beilianghleqthmaineq}
  {\rm GDet}_{{\rm P}}\left(A_{j \cdot}(b\cdot a_{i \cdot})\right)=  {\rm GDet}_{{\rm P}}\left(A_{\cdot j}(a_{\cdot i} \cdot b)\right)
  =
0.
\end{equation}
\end{theorem}

\begin{proof}
\begin{enumerate}[(1)]
  \item
If $b=0$, then  \eqref{beilianghleqthmaineq} holds by Item \eqref{basicproPgdecThitem1} in Theorem \ref{basicproPgdecTh}.

  \item
If $b\neq 0$. Let
\begin{equation*}
  A=[a_{xy}],~B=A_{j\cdot}( b \cdot  a_{i\cdot})=[b_{xy}],~C=A_{\cdot j}(a_{\cdot i } \cdot  b )=[C_{xy}],
\end{equation*}
for all $x,y\in[n]$.
Note that since $A=A^*$, we have
\[
b_{xy}=
\begin{cases}
b \cdot a_{iy} = b \cdot b_{iy}, & x=j,\ y \in [n],\\
a_{xy}=\overline{a_{yx}}=\overline{b_{yx}}, & x\ne j,\ y\ne j,\\
a_{xy}, & x\ne j,\ y=j,
\end{cases}
\quad
c_{xy}=
\begin{cases}
a_{xi} \cdot  b  =c_{xi} \cdot  b , & x \in [n], \ y=j, \\
a_{xy}=\overline{a_{yx}}=\overline{c_{yx}}, & x\ne j, \ y\ne j, \\
a_{xy},   &  x=j,\   y\ne j.
\end{cases}
\]
Thus, the proof of \eqref{beilianghleqthmaineq} is similar to that of
\eqref{GDetPAjairowcolmaineq} in Theorem \ref{ideteoctworeqth}, and hence is omitted.
\end{enumerate}
\end{proof}

\begin{remark}\label{Pljlirlricedrdeteqmark}
 Theorem \ref{ideteoctworeqth} is the special case of Theorem \ref{beilianghleqth} with \(b=1\). The rules  \({\rm P}_{l\text{-}j}\), \({\rm P}_{l\text{-}i}\), \({\rm P}_{r\text{-}j}\), and \({\rm P}_{r\text{-}i}\) satisfy the conditions on the rule \({\rm P}\) in Theorem \ref{ideteoctworeqth}.
 Let \(A=A^*\in\mathbb{H}^{n\times n}\), \(i,j\in[n]\) with \(i\neq j\), and $b\in \mathbb{H}$.
 Using Theorem \ref{beilianghleqth}, we obtain
\begin{align*}
  &{\rm GDet}_{{\rm P}_{l\text{-}j}}\left(A_{j\cdot}(b\cdot a_{i\cdot})\right)
   ={\rm GDet}_{{\rm P}_{l\text{-}j}}\left(A_{\cdot j}(a_{\cdot i}\cdot b)\right)
   ={\rm GDet}_{{\rm P}_{l\text{-}i}}\left(A_{j\cdot}(b\cdot a_{i\cdot})\right)
   ={\rm GDet}_{{\rm P}_{l\text{-}i}}\left(A_{\cdot j}(a_{\cdot i}\cdot b)\right)\\
  =&{\rm GDet}_{{\rm P}_{r\text{-}j}}\left(A_{j\cdot}(b\cdot a_{i\cdot})\right)
   ={\rm GDet}_{{\rm P}_{r\text{-}j}}\left(A_{\cdot j}(a_{\cdot i}\cdot b)\right)
   ={\rm GDet}_{{\rm P}_{r\text{-}i}}\left(A_{j\cdot}(b\cdot a_{i\cdot})\right)
   ={\rm GDet}_{{\rm P}_{r\text{-}i}}\left(A_{\cdot j}(a_{\cdot i}\cdot b)\right)
   =0.
\end{align*}
Then, by Items \eqref{PGdetspecaseremarkitem4} and \eqref{PGdetspecaseremarkitem5}
in Remark \ref{PGdetspecaseremark}, we obtain \cite[Theorems 4.3, 4.4]{kyrchei2008cramer},
i.e.,
\begin{align*}
 {\rm rdet}_j\left(A_{j\cdot}(b\cdot a_{i\cdot})\right)
 &={\rm rdet}_j\left(A_{\cdot j}(a_{\cdot i}\cdot b)\right)
 ={\rm cdet}_j\left(A_{j\cdot}(b\cdot a_{i\cdot})\right)
 ={\rm cdet}_j\left(A_{\cdot j}(a_{\cdot i}\cdot b)\right)
 =0,
\end{align*}
and some new results supplement the relevant properties of row-column
determinants, i.e.,
\begin{equation*}
{\rm rdet}_i\left(A_{j\cdot}(b\cdot a_{i\cdot})\right)
={\rm rdet}_i\left(A_{\cdot j}(a_{\cdot i}\cdot b)\right)
={\rm cdet}_i\left(A_{j\cdot}(b\cdot a_{i\cdot})\right)
={\rm cdet}_i\left(A_{\cdot j}(a_{\cdot i}\cdot b)\right)=0.
\end{equation*}
\end{remark}

The following theorem presents the related results for the ${\rm P}$-generalized determinant of a quaternion Hermitian matrix under one-sided elementary row (column)  operations of addition  and scaling.

\begin{theorem}\label{danbianXBHth}
Let \(A=A^*\in\mathbb{H}^{n\times n}\),  \(i,j\in[n]\) with \(i\neq j\), and \(b\in\mathbb{H}\).
\begin{enumerate}[(1)]
  \item\label{danbianXBHthXitem} Let the rule \(\rm P\) be given by Theorem \ref{ideteoctworeqth}. Then
\begin{equation*}
  {\rm GDet}_{\rm P}(P_{ij}(b)A)={\rm GDet}_{\rm P}(AP_{ij}(b))={\rm Det}(A).
\end{equation*}

  \item\label{danbianXBHthBitem} Let the rule \(\rm P\) satisfy the following:
      \begin{itemize}
        \item If \({\rm P}={\rm P}_l(S_n,\operatorname{first},\prec_{\triangleleft})\), then \(\operatorname{first}(Q)=i\) for every \(Q\in\mathcal{Q}([n])\) with \(i\in Q\);
        \item If \({\rm P}={\rm P}_r(S_n,\operatorname{last},\succ_{\triangleright})\), then \(\operatorname{last}(Q)=i\) for every \(Q\in\mathcal{Q}([n])\) with \(i\in Q\).
      \end{itemize}
   Then
  \begin{equation*}
    {\rm GDet}_{\rm P}\left(A_{i\cdot}(b\cdot a_{i\cdot})\right)={\rm GDet}_{\rm P}\left(A_{\cdot i}(a_{\cdot i}\cdot b)\right)=b\cdot{\rm Det}(A)={\rm Det}(A)\cdot b.
  \end{equation*}

\end{enumerate}
\end{theorem}

\begin{proof}
Let $A=[a_{xy}]$ for all $x,y \in [n]$.
\begin{enumerate}[(1)]
  \item\label{danbianXBHthproofitem1}
  Let $B= P_{ij}(b)A = [b_{xy}]$ for all $x,y \in [n]$, where
\begin{equation*}
  b_{iy} = a_{iy} + b a_{jy}
\end{equation*}
for all $y\in [n]$. Using Item \eqref{basicproPgdecThitem3} in Theorem \ref{basicproPgdecTh} and Theorem \ref{beilianghleqth}, we obtain
\begin{equation*}
    {\rm GDet}_{\rm P}(B)= {\rm GDet}_{\rm P}(A) + {\rm GDet}_{\rm P}\big(A_{i \cdot}(b \cdot a_{j \cdot} )\big) = {\rm GDet}_{\rm P}(A)={\rm Det}(A).
\end{equation*}
Similarly, we obtain ${\rm GDet}_{\rm P}(AP_{ij}(b))={\rm Det}(A)$.

  \item\label{danbianXBHthproofitem2}
 Let \({\rm P}={\rm P}_l(S_n,\operatorname{first},\prec_{\triangleleft})\), where  the selector $\operatorname{first}$ satisfies \(\operatorname{first}(Q)=i\) for every \(Q\in\mathcal{Q}([n])\) with \(i\in Q\).
 Let $B= A_{i\cdot}(b\cdot a_{i\cdot})=[b_{xy}]$ for all $x,y \in [n]$, where
\begin{equation}\label{bzyazybiybaiyeq}
  b_{zy}=a_{zy}  \quad \text{and} \quad  b_{iy} = b\cdot a_{iy}
\end{equation}
for all $z\in [n]\setminus \{i\}$ and $y \in [n]$.
Let the left-order cycle decomposition of a permutation $\sigma \in S_n$ under the rule ${\rm P}_l$ be given by \eqref{Plsigamlcfeneq}, and let $ \sigma _{k_{I}} \in \{ \sigma_{k_1},\sigma_{k_2},\ldots,\sigma_{k_u}\}$ satisfy $i \in \operatorname{supp} (\sigma _{k_I})$. Then, by \eqref{bzyazybiybaiyeq}, we have
\begin{align*}
f_{A}(\sigma_l)&= f_{A}^c (\sigma _{k_1} )\cdots f_{A}^c (\sigma _{k_I} )\cdots f_{A}^c (\sigma _{k_u} ),\\
f_{B}(\sigma_l)&= f_{B}^c (\sigma _{k_1} )\cdots f_{B}^c (\sigma _{k_I} )\cdots f_{B}^c (\sigma _{k_u} )=  f_{A}^c (\sigma _{k_1} )\cdots b \cdot f_{A}^c (\sigma _{k_I} )\cdots f_{A}^c (\sigma _{k_u} ).
\end{align*}
Let $T_{\sigma}=\{\sigma_l\}$ if $v_d\leq1$ for all $d=1,\ldots,u$, and
let $T_{\sigma}$ be given in \eqref{Tsigmaset2peq} if some $v_d\geq2$.
Then,  we claim that
     \begin{equation}\label{danbianXBHthproofitemIIImaineqsult}
          \sum_{\tau_l\in T_{\sigma}}{\rm sgn}(\tau)f_B(\tau_l)
=b    \sum_{\tau_l\in T_{\sigma}}{\rm sgn}(\tau)f_A(\tau_l).
       \end{equation}
\begin{enumerate}[(I)]
  \item\label{danbianXBHthproofitemI} Suppose $v_d\leq 1$ for all $d=1,2,\ldots,u$. Then, using \eqref{fAsigkdinReq}, we have $f_A^c(\sigma_{k_d}) \in \mathbb{R}$ for all $d=1,\ldots,u$. Therefore
  \begin{equation*}
f_{B}(\sigma_l) =   f_{A}^c (\sigma _{k_1} )\cdots b \cdot f_{A}^c (\sigma _{k_I} )\cdots f_{A}^c (\sigma _{k_u} ) = b \cdot f_{A}^c (\sigma _{k_1} )\cdots  f_{A}^c (\sigma _{k_I} )\cdots f_{A}^c (\sigma _{k_u} )
=   b \cdot f_{A}(\sigma_l),
  \end{equation*}
i.e., \eqref{danbianXBHthproofitemIIImaineqsult} holds.
  \item\label{danbianXBHthproofitemII} Assume $v_{w_1},\ldots,v_{w_p}\geq2$, where $p\geq1$ and ${w_1},\ldots,{w_p}\in\{1,\ldots,u\}$ satisfy
$
     i_{k_{w_1}}\triangleleft\ldots\triangleleft i_{k_{w_p}}
$.
From  Item \eqref{firstitemii} in the proof of Theorem \ref{HerPGdetmainproth}, it follows that
       \begin{equation*}
          \sum_{\tau_l\in T_{\sigma}}{\rm sgn}(\tau)f_A(\tau_l)
=(-1)^{n-u}\alpha\,{\rm tr}(f_A^c(\tau_{k_{w_1}}))\ldots{\rm tr}(f_A^c(\tau_{k_{w_p}}))\in\mathbb{R},
       \end{equation*}
where $\alpha$ is given in \eqref{alpha12channeq}.
\begin{enumerate}[(i)]
  \item\label{danbianXBHthproofitemi} If the index $I$ satisfies $ I \notin \{{w_1},\ldots,{w_p}\}$, then
  \begin{equation*}
    \sum_{\tau_l\in T_{\sigma}}{\rm sgn}(\tau)f_B(\tau_l)
=(-1)^{n-u}\, b\,\alpha\,{\rm tr}(f_A^c(\tau_{k_{w_1}}))\ldots{\rm tr}(f_A^c(\tau_{k_{w_p}}))=
b \sum_{\tau_l\in T_{\sigma}}{\rm sgn}(\tau)f_A(\tau_l).
\end{equation*}

\item\label{danbianXBHthproofitemii} If the index $I$ satisfies $ I \in \{{w_1},\ldots,{w_p}\}$, then
using \eqref{bzyazybiybaiyeq} and  \eqref{Fcongongesigameq}, we have
\begin{equation*}
f_B^c(  \overline{\sigma_{k_{I}}})={b} \cdot   f_A^c( \overline{\sigma_{k_{I}}}) =  {b} \cdot \overline{f_A^c(\sigma_{k_{I}})},
\end{equation*}
  where $\overline{\sigma_{k_{I}}}=(i_{k_{I}},i_{k_{I}+v_{I}},\ldots,i_{k_{I}+1})$.
Thus, similar to \eqref{sigmaRnueq}, we obtain
     \begin{equation*}
          \sum_{\tau_l\in T_{\sigma}}{\rm sgn}(\tau)f_B(\tau_l)
=b\, \left((-1)^{n-u}\alpha\,{\rm tr}(f_A^c(\tau_{k_{w_1}}))\ldots{\rm tr}(f_A^c(\tau_{k_{w_p}}))\right) =b    \sum_{\tau_l\in T_{\sigma}}{\rm sgn}(\tau)f_A(\tau_l).
       \end{equation*}

\end{enumerate}
\end{enumerate}

Hence,  using \eqref{danbianXBHthproofitemIIImaineqsult} and \eqref{GDetPlAsumdaegeq}, we obtain
 \begin{equation*}
    {\rm GDet}_{{\rm P}_l}\left(A_{i\cdot}(b\cdot a_{i\cdot})\right)=    {\rm GDet}_{{\rm P}_l}\left(B\right) = b\cdot {\rm GDet}_{{\rm P}_l}\left(A\right) = b\cdot{\rm Det}(A),
 \end{equation*}
 which, together with  Theorem \ref{HerPGdetmainproth}, yields
\begin{equation}\label{GdetplAibaieqbdeatAdetbeq}
 {\rm GDet}_{{\rm P}_l}\left(A_{i\cdot}(b\cdot a_{i\cdot})\right)= b  \cdot {\rm Det}(A) ={\rm Det}(A)\cdot b .
\end{equation}
  If \({\rm P}={\rm P}_r(S_n,\operatorname{last},\succ_{\triangleright})\) and the selector $\operatorname{last}$ satisfies  \(\operatorname{last}(Q)=i\) for every \(Q\in\mathcal{Q}([n])\) with \(i\in Q\), similar to the proof of \eqref{GdetplAibaieqbdeatAdetbeq}, we can derive
 \begin{equation}\label{GdetplAibaieqbdeatAdetbringteq}
 {\rm GDet}_{{\rm P}_r}\left(A_{i\cdot}(b\cdot a_{i\cdot})\right)=  b  \cdot {\rm Det}(A) = {\rm Det}(A)\cdot b.
\end{equation}
Therefore,  combining  \eqref{GdetplAibaieqbdeatAdetbeq} and \eqref{GdetplAibaieqbdeatAdetbringteq},  it follows
\begin{equation*}
  {\rm GDet}_{\rm P}\left(A_{ i\cdot}(b \cdot  a_{ i\cdot} )\right)=b\cdot{\rm Det}(A)={\rm Det}(A) \cdot b.
\end{equation*}
Furthermore,   using Theorem \ref{genGDetlGdetRlinkth}, we obtain
\begin{equation*}
  {\rm GDet}_{\rm P}\left(A_{\cdot i}(a_{\cdot i}\cdot b)\right)=b\cdot{\rm Det}(A)={\rm Det}(A) \cdot b.
\end{equation*}
This finishes  the proof.
\end{enumerate}
\end{proof}

\begin{remark}
The rule ${\rm P}$ given in Item \eqref{danbianXBHthBitem} of Theorem
\ref{danbianXBHth} satisfies the restriction on the rule ${\rm P}$ in
Theorem \ref{ideteoctworeqth}.
Since ${\rm P}_{l\text{-}i}$ and ${\rm P}_{r\text{-}i}$ are special cases
of the rule ${\rm P}$ given in Item \eqref{danbianXBHthBitem} of Theorem
\ref{danbianXBHth}, by Theorem \ref{danbianXBHth} we obtain
\begin{align*}
    &  {\rm GDet}_{{\rm P}_{l\text{-}i}}\left(A_{i\cdot}(b\cdot a_{i \cdot})\right)={\rm GDet}_{{\rm P}_{l\text{-}i}}\left( A_{\cdot i}(a_{\cdot i} \cdot b) \right)={\rm Det}(A) \cdot  b
    \\ = &{\rm GDet}_{{\rm P}_{r\text{-}i}}\left(A_{i\cdot}(b\cdot a_{i \cdot})\right)
 = {\rm GDet}_{{\rm P}_{r\text{-}i}}\left( A_{\cdot i}(a_{\cdot i} \cdot b) \right)
=b \cdot {\rm Det}(A),
\end{align*}
and further derive
\cite[Lemmas 5.1, 5.2]{kyrchei2012theory},  i.e.,
\begin{align*}
&{\rm rdet}_{i} \left(A_{i\cdot}(b\cdot a_{i\cdot})\right)  =
  {\rm rdet}_{i}\left(A_{\cdot i}(a_{\cdot i}\cdot b)\right)={\rm det}(A)\cdot b
    \\ = & {\rm cdet}_{i}\left(A_{i\cdot}(b\cdot a_{i\cdot})\right)
    = {\rm cdet}_{i}\left(A_{\cdot i}(a_{\cdot i}\cdot b)\right)
  = b\cdot{\rm det}(A)
  .
\end{align*}
Theorem \ref{danbianXBHth}  discusses the related conclusions for the elementary row (column) operations of addition and scaling on quaternion Hermitian matrices.
However, the situation is different for the third operation,
interchanging.
In general, interchanging two rows or columns of a Hermitian matrix does not
simply reverse the sign of the determinant.
 For example,
let
\[
A= \begin{bmatrix}
1 & \mathbf{i} & \mathbf{j}\\
-\mathbf{i} & 1 & \mathbf{k}\\
-\mathbf{j} & -\mathbf{k} & 1
\end{bmatrix}
\quad\text{and}\quad
B=\begin{bmatrix}
1 & \mathbf{i} & \mathbf{j}\\
-\mathbf{j} & -\mathbf{k} & 1\\
-\mathbf{i} & 1 & \mathbf{k}
\end{bmatrix}.
\]
 Then, $A\in\mathbb{H}^{3\times 3}$ is Hermitian,
 and $B\in\mathbb{H}^{3\times 3}$ is  obtained by interchanging  the second and third rows of $A$.
Since row-column determinants are representative in the ${\rm P}$-generalized determinant setting,  direct computation gives
\begin{align*}
 &  {\rm Det}(A) =   {\rm rdet}_1(A)= {\rm rdet}_2(A)= {\rm rdet}_3(A)
  =  {\rm cdet}_1(A)= {\rm cdet}_2(A)= {\rm cdet}_3(A)=-4, \\
  &   {\rm rdet}_1(B)= 0, ~{\rm rdet}_2(B)=   {\rm rdet}_3(B)= 2, ~    {\rm cdet}_1(B)=0,~   {\rm cdet}_2(B)=   {\rm cdet}_3(B)=2.
\end{align*}
Therefor, after interchanging two rows or columns of a Hermitian quaternion  matrix,  the determinant may change either in sign or in value.
\end{remark}

The following corollary presents the properties of the ${{\rm P}_{l\text{-}j}}$- and ${\rm P}_{r\text{-}j}$-generalized determinants when a row (column) of a quaternion Hermitian matrix is a left (right) linear combination of the other rows (columns).

\begin{corollary}\label{lineaproHerPgdetth}
Let $A = A^* \in \mathbb{H}^{n\times n}$, and $b_k\in \mathbb{H}$ for $k=1,2,\ldots,n$. Then,
\begin{enumerate}[(1)]
  \item
 $
  {\rm GDet}_{{\rm P}_{l\text{-}j}}\left(A_{j\cdot}\left(\sum\limits_{k=1,k\neq j}^{n} b_k\cdot a_{k\cdot}\right)\right)=  {\rm GDet}_{{\rm P}_{l\text{-}j}}\left(A_{\cdot j}\left(\sum\limits_{k=1,k\neq j}^{n} a_{\cdot k} \cdot b_k \right)\right)=0$;

  \item
$
{\rm GDet}_{{\rm P}_{l\text{-}j}}\left(A_{j\cdot}\left(\sum\limits_{k=1}^{n} b_k\cdot a_{k\cdot}\right)\right)=  {\rm GDet}_{{\rm P}_{l\text{-}j}}\left(A_{\cdot j}\left(\sum\limits_{k=1}^{n} a_{\cdot k} \cdot b_k \right)\right)
  =b_j\cdot {\rm Det}(A)
$;

  \item
 $
  {\rm GDet}_{{\rm P}_{r\text{-}j}}\left(A_{j\cdot}\left(\sum\limits_{k=1,k\neq j}^{n} b_k\cdot a_{k\cdot}\right)\right)=  {\rm GDet}_{{\rm P}_{r\text{-}j}}\left(A_{\cdot j}\left(\sum\limits_{k=1,k\neq j}^{n} a_{\cdot k} \cdot b_k \right)\right)=0$;

  \item
$
{\rm GDet}_{{\rm P}_{r\text{-}j}}\left(A_{j\cdot}\left(\sum\limits_{k=1}^{n} b_k\cdot a_{k\cdot}\right)\right)=  {\rm GDet}_{{\rm P}_{r\text{-}j}}\left(A_{\cdot j}\left(\sum\limits_{k=1}^{n} a_{\cdot k} \cdot b_k \right)\right)
  =b_j\cdot {\rm Det}(A).
$

\end{enumerate}
\end{corollary}

\begin{proof}
These assertions follow directly from Theorems \ref{basicproPgdecTh}, \ref{beilianghleqth}, and \ref{danbianXBHth}.
\end{proof}

\begin{remark}
By Items \eqref{PGdetspecaseremarkitem4} and \eqref{PGdetspecaseremarkitem5} in Remark \ref{PGdetspecaseremark}, Corollary \ref{lineaproHerPgdetth} yields \cite[Theorems 5.7--5.10]{kyrchei2012theory}, i.e.,
\begin{enumerate}[(1)]
  \item ${\rm rdet}_j\left(A_{j\cdot}\left(\sum\limits_{k=1,k\neq j}^{n} b_k\cdot a_{k \cdot}\right)\right)={\rm rdet}_j\left(A_{\cdot j}\left(\sum\limits_{k=1,k\neq j}^{n} a_{\cdot k } \cdot  b_k \right)\right)=0$;

  \item ${\rm rdet}_j\left(A_{j\cdot}\left( a_{j\cdot} + \sum\limits_{k=1,k\neq j}^{n} b_k\cdot a_{k \cdot}\right)\right)=
      {\rm rdet}_j\left(A_{\cdot j}\left(
  a_{\cdot j} + \sum\limits_{k=1,k\neq j}^{n} a_{ \cdot k} \cdot  b_k \right)\right)= {\rm det}(A)$;

  \item ${\rm cdet}_j\left(A_{j\cdot}\left(\sum\limits_{k=1,k\neq j}^{n} b_k\cdot a_{k \cdot}\right)\right)={\rm cdet}_j\left(A_{\cdot j}\left(\sum\limits_{k=1,k\neq j}^{n} a_{\cdot k } \cdot  b_k \right)\right)=0$;

  \item ${\rm cdet}_j\left(A_{j\cdot}\left( a_{j\cdot} + \sum\limits_{k=1,k\neq j}^{n} b_k\cdot a_{k \cdot}\right)\right)=
      {\rm cdet}_j\left(A_{\cdot j}\left(
  a_{\cdot j} + \sum\limits_{k=1,k\neq j}^{n} a_{ \cdot k} \cdot  b_k \right)\right)= {\rm det}(A)$.
\end{enumerate}
\end{remark}

\begin{remark}\label{Shuangproremark}
Let $A = A^* \in \mathbb{H}^{n\times n}$ and $b\in \mathbb{H}$.
Using Theorem \ref{HerPGdetmainproth}, Remark \ref{GPdeteqrcdeteqchendetremark},  and some results for the Chen  and row-column determinants, we obtain the following properties:
\begin{enumerate}[(1)]
  \item
  $
  {\rm Det}\left( P_{ij}(b)A (P_{ij}(b))^* \right)= {\rm Det}(A)
  $ { (by  \cite[Theorem 7.1]{kyrchei2008cramer} or  \cite[Theorem 2]{chen1991inverse})};

  \item
  $
  {\rm Det}\left( P_{i}(b)A (P_{i}(b))^* \right)= \vert b \vert ^2 {\rm Det}(A)
  $
(by \cite[Theorem 1]{li1995properties});

  \item
  $
  {\rm Det}\left( P_{ij}A (P_{ij})^* \right)={\rm Det}(A)
  $
 (by \cite[Theorem 1.2]{chen1991definition});

\item\label{Shuangproremarkitem004} ${\rm Det}(A) \neq 0$ if and only if $A$ is invertible
 (by \cite[Theorem 7.3]{kyrchei2012theory});

\item\label{Shuangproremarkitem4}
  $
  {\rm Det}\left( UA U^* \right)= {\rm Det}(A),
  $
  where $U\in \mathbb{H}^{n\times n}$ is a unitary matrix
 (by  \cite[Theorem 3]{li1995properties});

  \item\label{Shuangproremarkitem5}
  $
  {\rm Det}(A)=\prod\limits_{i=1}^{n} \lambda_{i},
  $
  where $\lambda_1,...,\lambda_n$  are the $n$ real right eigenvalues of $A$  (by  \cite[Theorem 7.3]{kyrchei2008cramer});

  \item\label{Shuangproremarkitem6}
  $
  {\rm Det}(AB)={\rm Det}(A){\rm Det}(B)
  $, where $A,B, AB\in \mathbb{H}^{n \times n}$ are all Hermitian  (by  \cite[Corollary 6.1]{kyrchei2012theory}).

\end{enumerate}

\end{remark}

\begin{remark}
We now examine how the $\mathrm{P}$-generalized determinant of quaternion Hermitian matrices relates to Axioms~1--3  in Section~\ref{Introductionsection}.
\begin{enumerate}[(1)]
  \item By Item \eqref{Shuangproremarkitem004} in Remark \ref{Shuangproremark}, for any rule ${\rm P}$, the ${\rm P}$-generalized determinant on quaternion Hermitian matrices satisfies Axiom 1.
  \item By Item \eqref{danbianXBHthXitem} of Theorem \ref{danbianXBHth}, if the rule ${\rm P}$ is given by Theorem \ref{ideteoctworeqth}, then the ${\rm P}$-generalized determinant  on quaternion Hermitian matrices  satisfies Axioms 1 and 3.
  \item By Item \eqref{Shuangproremarkitem6} of Remark \ref{Shuangproremark}, if the rule ${\rm P}$ is given by Theorem \ref{ideteoctworeqth} and the product of Hermitian matrices is required to remain Hermitian, then the ${\rm P}$-generalized determinant  on quaternion Hermitian matrices  satisfies Axioms 1, 2, and 3.
\end{enumerate}
\end{remark}

\section{Applications of ${\rm P}$-generalized determinant}\label{appPgenrdetssection}
Building on the properties of the \({\rm P}\)-generalized determinant established in Sections~\ref{Pgenrdetsdefsection} and \ref{PgenrdetHersection}, we begin this section by giving the determinantal representation of the inverse of a quaternion Hermitian matrix.

\begin{theorem}\label{quaHerInmafenth}
Let $A = A^* \in \mathbb{H}^{n\times n}$ with ${\rm Det}(A) \neq 0$,  and let the rules $\operatorname{P}_{l\text{-}i}$ and $\operatorname{P}_{r\text{-}j}$ be given for  all $i,j\in[n]$.
For any $i\in[n]$ and each $j\in[n]_{i} :=[n]\setminus\{i\}$,
let $\operatorname{P}_{l\text{-}j}^{(i)}$ denote the left-order cycle representation rule  given by \eqref{cofatortheoremplgj}.
For any $j\in[n]$ and each $i\in [n]_{j} := [n]\setminus\{j\}$, let $\operatorname{P}_{r\text{-}i}^{(j)}$ denote the right-order cycle representation rule given by \eqref{cofatortheoremprgi}.
Let
\begin{equation*}
     R_A= \frac{1}{\operatorname{Det}(A)}\left[
            \begin{array}{cccc}
          R_{11}  & R_{21}  & \cdots  & R_{n1} \\
          R_{12} & R_{22}  & \cdots  & R_{n2}\\
       \vdots  & \vdots  &   \cdots  & \vdots\\
        R_{1n}  & R_{2n}  & \cdots  & R_{nn}
            \end{array}
          \right],\quad
        L_A   =\frac{1}{\operatorname{Det}(A)}
           \left[
            \begin{array}{cccc}
          L_{11}  & L_{21}  & \cdots  & L_{n1} \\
          L_{12} & L_{22}  & \cdots  & L_{n2}\\
       \vdots  & \vdots  &   \cdots  & \vdots\\
        L_{1n}  & L_{2n}  & \cdots  & L_{nn}
            \end{array}
          \right],
\end{equation*}
where
\begin{equation*}
 R_{ij}=\begin{cases}
   \operatorname{Det}(A^{i,i}), & j=i,\\[1mm]
  -\operatorname{GDet}_{{\operatorname{P}_{l\text{-}j}^{(i)}}}\left((A_{\cdot j}(a_{\cdot i}))^{i,i}\right), & j \neq i,
    \end{cases}\quad
    L_{ij}=\begin{cases}
   \operatorname{Det}(A^{j,j}), & i=j,\\[1mm]
   -\operatorname{GDet}_{\operatorname{P}_{r\text{-}i}^{(j)}}\left((A_{i\cdot}(a_{j\cdot}))^{j,j}\right), & i \neq j.
    \end{cases}
\end{equation*}
Then
\begin{equation*}
A^{-1}= R_A=L_A.
\end{equation*}
\end{theorem}

\begin{proof}
Let $A=[a_{xy}]$ for all $x,y\in[n]$. Using Theorems  \ref{cofatortheorem} and  \ref{HerPGdetmainproth}, we obtain
\begin{equation*}
 \sum_{k=1}^{n}a_{ik}R_{ik} = {\rm GDet}_{{\rm P}_{l\text{-}i}}(A) = {\rm Det}(A),
\end{equation*}
where $i\in[n]$.
Using Theorems  \ref{cofatortheorem} and \ref{ideteoctworeqth}, we get
\begin{equation*}
   \sum_{k=1}^{n}a_{ik}R_{jk} = {\rm GDet}_{{\rm P}_{l\text{-}j}}(A_{j\cdot}(a_{i\cdot})) = 0,
\end{equation*}
where $i,j\in[n]$ with $i\neq j$.
Therefore,
\begin{equation*}
  A R_A = \frac{1}{{\rm Det}(A)}
  \left[
    \begin{array}{ccc}
       \sum\limits_{k=1}^{n}a_{1k}R_{1k}    & \cdots &   \sum\limits_{k=1}^{n}a_{1k}R_{nk}  \\
    \vdots    & \vdots  &  \vdots \\
       \sum\limits_{k=1}^{n}a_{nk}R_{1k}    & \cdots &   \sum\limits_{k=1}^{n}a_{nk}R_{nk}  \\
    \end{array}
  \right]
  = \frac{1}{{\rm Det}(A)}
  \left[
    \begin{array}{ccc}
   {\rm Det}(A)  & \cdots &  0  \\
    \vdots    &  \ddots  &  \vdots \\
 0  & \cdots  &     {\rm Det}(A)  \\
    \end{array}
  \right]
  = I_n,
\end{equation*}
i.e., $A^{-1}=R_A$.
By a similar argument, it follows that \(A^{-1}=L_A\).
\end{proof}

Ivan \cite[Remark 7.1]{kyrchei2012theory} defines the classical adjoint matrix of an invertible quaternion Hermitian matrix via row-column determinants, denoted by \(\operatorname{Adj}[A]\).
Motivated by this, we now  introduce the ${\rm P}$-adjoint matrix  based on the \({\rm P}\)-generalized determinant.

\begin{remark}
 For a quaternion Hermitian matrix \(A\in\mathbb{H}^{n\times n}\) with \({\rm Det}(A)\neq 0\) and \(i,j\in[n]\) with \(i\neq j\),   using Theorem \ref{quaHerInmafenth}  we have that
\begin{equation*}
   R_{ij}:= -{\rm GDet}_{{\rm P}_{l\text{-}j}^{(i)}} \left((A_{\cdot j}(a_{\cdot i}))^{i,i}\right) \quad \text{and} \quad
     L_{ij}:= -{\rm GDet}_{{\rm P}_{r\text{-}i}^{(j)}}\left((A_{i\cdot}(a_{j\cdot}))^{j,j}\right)
\end{equation*}
 are equal and this value is independent of the rules ${\rm P}_{l\text{-}i}$ and ${\rm P}_{r\text{-}j}$.
 Hence, under the hypotheses of Theorem~\ref{quaHerInmafenth},
 we define the \(\mathrm{P}\)-adjoint matrix of \(A\), denoted by \({\rm Adj}_{{\rm P}}\), as follows:
\begin{equation}\label{Padjmatrixdefeq}
     {\rm Adj}_{{\rm P}}(A)= \left[
 \begin{array}{cccc}
          R_{11}  & R_{21}  & \ldots  & R_{n1} \\
          R_{12} & R_{22}  & \ldots  & R_{n2}\\
       \vdots  & \vdots  &   \ldots  & \vdots\\
        R_{1n}  & R_{2n}  & \ldots  & R_{nn}
            \end{array}
          \right]
    =
    \left[
            \begin{array}{cccc}
          L_{11}  & L_{21}  & \ldots  & L_{n1} \\
          L_{12} & L_{22}  & \ldots  & L_{n2}\\
       \vdots  & \vdots  &   \ldots  & \vdots\\
        L_{1n}  & L_{2n}  & \ldots  & L_{nn}
            \end{array}
          \right].
\end{equation}
 Furthermore, from  Theorem \ref{quaHerInmafenth},  \cite[Remark 7.1]{kyrchei2012theory},  and Remark \ref{GPdeteqrcdeteqchendetremark}, it follows that
\begin{equation}\label{Padjmatrixproeq}
     {\rm Adj}_{{\rm P}}(A)\cdot A=A\cdot{\rm Adj}_{{\rm P}}(A)={\rm Det}(A)\cdot I_n, \quad
     {\rm Det}({\rm Adj}_{\rm P}(A))=({\rm Det}(A))^{n-1}.
\end{equation}
In particular, if \({\rm P}_{l\text{-}j}\) and \({\rm P}_{r\text{-}i}\) satisfy \eqref{PGdetspecaseremarkitem4} and \eqref{PGdetspecaseremarkitem5} of Remark \ref{PGdetspecaseremark}, respectively, then the \({\rm P}\)-adjoint matrix of \(A\) coincides with its classical adjoint matrix \({\rm Adj}[A]\).
\end{remark}

Applying Theorem~\ref{quaHerInmafenth} on the determinantal representation of the inverse of an invertible quaternion Hermitian matrix, we derive Cramer's rule for solving matrix equations over  quaternions.

\begin{theorem}\label{GDETPmatreqappth}
Let \(A = A^* \in \mathbb{H}^{n\times n}\), \(B = B^* \in \mathbb{H}^{m\times m}\), and \(C \in \mathbb{H}^{n\times m}\), with \({\rm Det}(A) \neq 0\) and \({\rm Det}(B) \neq 0\).
 Let $c_k$ and $e_k$ denote the $k$-th column vectors of $C$ and $I_m$, respectively, for $k=1,\ldots,m$.
Consider the quaternion matrix equation
\begin{equation}\label{fenbiaoappmaeqeq}
  AXB=C,
\end{equation}
where $X \in \mathbb{H}^{n\times m}$ is unknown. Then the components of the unique solution $X=[x_{ij}]\in \mathbb{H}^{n\times m}$ of Equation \eqref{fenbiaoappmaeqeq} are given by
\begin{equation*}
   x_{ij}= \frac{\sum\limits_{k=1}^{m}  {\rm GDet}_{{\rm P}_{r\text{-}i}} \left(A_{\cdot i}(c_k)\right)  \cdot  {\rm GDet}_{{\rm P}_{l\text{-}j}}
    \left(B_{j \cdot}(e_{k}^T) \right)    } {{\rm Det}(A){\rm Det}(B)},
\end{equation*}
where $i=1,\ldots,n$ and $j=1,\ldots,m$.
\end{theorem}

\begin{proof}
Let $C=[c_{xy}]\in \mathbb{H}^{n\times m}$ and $I_{m}=[e_{xy}]\in \mathbb{H}^{m\times m}$.
By \eqref{Padjmatrixproeq} and \eqref{Padjmatrixdefeq}, it follows that
\begin{equation*}
  A^{-1}= \alpha  {\rm Adj}_{{\rm P}}(A)=  \alpha
   \left[
            \begin{array}{ccc}
          L_{11}  &  \ldots  & L_{n1} \\
       \vdots  &  \ldots & \vdots\\
        L_{1n}  & \ldots  & L_{nn}
            \end{array}
          \right], ~
B^{-1} = \beta  {\rm Adj}_{{\rm P}}(B)   = \beta
\left[
            \begin{array}{ccc}
          R_{11}  &  \ldots  & R_{m1} \\
       \vdots  &  \ldots & \vdots\\
        R_{1m}  & \ldots  & R_{mm}
            \end{array}
          \right],
\end{equation*}
where \(\alpha = \frac{1}{\operatorname{Det}(A)}\) and \(\beta = \frac{1}{\operatorname{Det}(B)}\).
Thus, using Theorem \ref{cofatortheorem}, we obtain
\begin{align*}
A^{-1} C&= \alpha\,
           \left[
            \begin{array}{ccc}
          L_{11}  &  \ldots  & L_{n1} \\
       \vdots  &  \ldots & \vdots\\
        L_{1n}  & \ldots  & L_{nn}
            \end{array}
          \right]
          \left[
            \begin{array}{ccc}
           c_1  &  \ldots &    c_m \\
            \end{array}
          \right]\\
 &=     \alpha\,   \left[
            \begin{array}{cccc}
        \sum\limits_{k=1}^{n} L_{k1}c_{k1}  &\ldots &    \sum\limits_{k=1}^{n} L_{k1}c_{km} \\
       \vdots  & \ldots  &   \vdots  \\
   \sum\limits_{k=1}^{n} L_{kn}c_{k1}   & \ldots& \sum\limits_{k=1}^{n} L_{kn}c_{km}
            \end{array}
          \right]
          = \alpha\,
                \left[
            \begin{array}{cccc}
       {\rm GDet}_{{\rm P}_{r\text{-}1}}(A_{\cdot 1}(c_{1})) &\ldots &    {\rm GDet}_{{\rm P}_{r\text{-}1}}(A_{\cdot 1}(c_{m})) \\
       \vdots  & \ldots  &   \vdots  \\
{\rm GDet}_{{\rm P}_{r\text{-}n}}(A_{\cdot n}(c_{1}))  & \ldots& {\rm GDet}_{{\rm P}_{r\text{-}n}}(A_{\cdot n}(c_{m}))
            \end{array}
          \right],
\end{align*}
\begin{align*}
B^{-1}& =I_{m}B^{-1} =
\beta\,  \left[
  \begin{array}{c}
    e_1^T  \\
    \vdots  \\
     e_m^T      \\
  \end{array}
\right]
\left[
            \begin{array}{ccc}
          R_{11}  &  \ldots  & R_{m1} \\
       \vdots  &  \ldots & \vdots\\
        R_{1m}  & \ldots  & R_{mm}
            \end{array}
          \right] \\
          & =
         \beta \,\left[
            \begin{array}{cccc}
     \sum\limits_{k=1}^{m}e_{1k} R_{1k}  &\ldots &    \sum\limits_{k=1}^{m} e_{1k} R_{mk}  \\
       \vdots  & \ldots  &   \vdots  \\
     \sum\limits_{k=1}^{m} e_{mk} R_{1k}   & \ldots&     \sum\limits_{k=1}^{m} e_{mk}R_{mk}
            \end{array}
          \right]
          =  \beta\,
    \left[
            \begin{array}{cccc}
       {\rm GDet}_{{\rm P}_{l\text{-}1}}(B_{ 1 \cdot}(e_{1}^T)) &\ldots &    {\rm GDet}_{{\rm P}_{l\text{-}m}}(B_{ m \cdot}(e_{1}^T)) \\
       \vdots  & \ldots  &   \vdots  \\
{\rm GDet}_{{\rm P}_{l\text{-}1}}(B_{ 1 \cdot}(e_{m}^T))  & \ldots& {\rm GDet}_{{\rm P}_{l\text{-}m}}(B_{ m \cdot}(e_{m}^T))
            \end{array}
          \right].
\end{align*}
Therefore,  the unique solution of Equation \eqref{fenbiaoappmaeqeq} is
\begin{align*}
 X& = A^{-1} CB^{-1}\\
& ={\footnotesize
\gamma\, \left[
            \begin{array}{cccc}
       {\rm GDet}_{{\rm P}_{r\text{-}1}}(A_{\cdot 1}(c_{1})) &\ldots &    {\rm GDet}_{{\rm P}_{r\text{-}1}}(A_{\cdot 1}(c_{m})) \\
       \vdots  & \ldots  &   \vdots  \\
{\rm GDet}_{{\rm P}_{r\text{-}n}}(A_{\cdot n}(c_{1}))  & \ldots& {\rm GDet}_{{\rm P}_{r\text{-}n}}(A_{\cdot n}(c_{m}))
            \end{array}
          \right]
          \left[
            \begin{array}{cccc}
       {\rm GDet}_{{\rm P}_{l\text{-}1}}(B_{ 1 \cdot}(e_{1}^T)) &\ldots &    {\rm GDet}_{{\rm P}_{l\text{-}m}}(B_{ m \cdot}(e_{1}^T)) \\
       \vdots  & \ldots  &   \vdots  \\
{\rm GDet}_{{\rm P}_{l\text{-}1}}(B_{ 1 \cdot}(e_{m}^T))  & \ldots& {\rm GDet}_{{\rm P}_{l\text{-}m}}(B_{ m \cdot}(e_{m}^T))
            \end{array}
          \right]} \\
 & =
 {\footnotesize \gamma\,
 \left[
            \begin{array}{cccc}
       \sum\limits_{k=1}^{m}  {\rm GDet}_{{\rm P}_{r\text{-}1}} \left(A_{\cdot 1}(c_k)\right)
        \cdot  {\rm GDet}_{{\rm P}_{l\text{-}1}}
    \left(B_{ 1 \cdot}(e_{k}^T) \right)
     &\ldots
    &    \sum\limits_{k=1}^{m}  {\rm GDet}_{{\rm P}_{r\text{-}1 }} \left(A_{\cdot 1}(c_k)\right)  \cdot  {\rm GDet}_{{\rm P}_{l\text{-}m}}
    \left(B_{ m \cdot}(e_{k}^T) \right)
     \\
       \vdots  & \ldots  &   \vdots  \\
 \sum\limits_{k=1}^{m}  {\rm GDet}_{{\rm P}_{r\text{-}n }} \left(A_{\cdot n}(c_k)\right)  \cdot  {\rm GDet}_{{\rm P}_{l\text{-}1}}
    \left(B_{ 1 \cdot}(e_{k}^T) \right)
    & \ldots& \sum\limits_{k=1}^{m}  {\rm GDet}_{{\rm P}_{r\text{-}n }} \left(A_{\cdot n}(c_k)\right)
     \cdot  {\rm GDet}_{{\rm P}_{l\text{-}m}}
    \left(B_{ m \cdot}(e_{k}^T) \right)
            \end{array}
          \right]},
\end{align*}
where $\gamma= \frac{1}{{\rm Det}(A){\rm Det}(B)}$.
\end{proof}

The following corollary is an immediate consequence of Theorem~\ref{GDETPmatreqappth}, giving Cramer's rule for the unique solutions of the quaternion linear systems in terms of the \({\rm P}\)-generalized determinant.

\begin{corollary}\label{caremrulineqth}
Let $A = A^* \in \mathbb{H}^{n\times n}$ with ${\rm Det}(A) \neq 0$.
\begin{enumerate}[(1)]
  \item Let $b\in \mathbb{H}^{n\times 1}$. Then the unique solution $x=[x_1, x_2,\ldots, x_n]^T \in \mathbb{H}^{n\times 1}$ of the system of linear equations
  \begin{equation*}
    Ax=b
  \end{equation*}
  has components given by
  \begin{equation*}
    x_j = \frac{{\rm GDet}_{{\rm P}_{r\text{-}j}}
    \left(A_{\cdot j}(b)\right)}
    {{\rm Det}(A)},
  \end{equation*}
  where $j=1,2,\ldots,n$.

  \item Let $b\in \mathbb{H}^{1\times n}$. Then the unique solution $x=[x_1, x_2,\ldots, x_n]\in  \mathbb{H}^{1\times n}$ of the system of linear equations
  \begin{equation*}
    xA=b
  \end{equation*}
  has components given by
  \begin{equation*}
    x_i = \frac{{\rm GDet}_{{\rm P}_{l\text{-}i}}
    \left(A_{i \cdot}(b) \right)}
    {{\rm Det}(A)},
  \end{equation*}
  where $i=1,2,\ldots,n$.
\end{enumerate}
\end{corollary}

\begin{remark}
In particular, if \({\rm P}_{r\text{-}j}\) and \({\rm P}_{i\text{-}i}\) are chosen to satisfy \eqref{PGdetspecaseremarkitem4} and \eqref{PGdetspecaseremarkitem5} of Remark \ref{PGdetspecaseremark}, respectively, then Theorem~\ref{GDETPmatreqappth} recovers \cite[Theorems 9.1, 9.2]{kyrchei2008cramer}, and Corollary~\ref{caremrulineqth} recovers \cite[Theorems 9.3, 9.4]{kyrchei2008cramer}.
\end{remark}

\section{Conclusions}
\label{conclusection}
 This paper introduces the \({\rm P}\)-generalized determinant of quaternion matrices.
 We investigate its equivalent characterization under row or column expansion, and
 the conversion between the \({\rm P}_{l}\)- and \({\rm P}_{r}\)-generalized determinants.
 We also study the invariance of the \({\rm P}\)-generalized determinant for quaternion Hermitian matrices,
 together with its behavior under row (column) replacement and one-sided elementary row (column) operations.
 Furthermore, we apply the \({\rm P}\)-generalized determinant to derive the determinantal representation of the inverse of an invertible quaternion Hermitian matrix,
 define the \({\rm P}\)-adjoint matrix,
 and establish Cramer's rule for quaternion matrix equations and linear systems.

The \({\rm P}\)-generalized determinant unifies the five classical quaternion determinants: \({\rm Mdet}\), \({\rm Xdet}\), \({\rm Cdet}\), \({\rm rdet}_{i}\), and \({\rm cdet}_{j}\).
It substantially relaxes the strict product-order restrictions inherent in the multivariate polynomial-type definitions of these classical determinants, while retaining most of their desirable properties
We believe  that the $\mathrm{P}$-generalized determinant opens up  rich prospects for future research, which may be explored along the following directions:

\begin{enumerate}[(1)]
  \item
  In recent years, dual quaternions \cite{clifford1873preliminary}, as an extension of quaternions, have been widely studied \cite{qi2022dual,qi2023eigenvalues}
  and effectively applied to formation control of multi-agent and multi-rigid-body systems such as
  uncrewed  aerial vehicles, robots, and satellites \cite{cui2024dual,liu2025distributed,chen2024dual,qi2023dualM}; the hand-eye calibration problem and the simultaneous localization and mapping problem \cite{qi2023standardD}; and 3D path planning    \cite{xia2026novel}.
Recently, Cui et al.\ \cite{cui2024moore} extended the $\rm Mdet$   and ${\rm Cdet}$  from quaternion Hermitian matrices to dual quaternion Hermitian matrices, and applied them to the gain graph theory \cite{cui2024spectral}.
Naturally, extending the $\mathrm{P}$-generalized determinant theory from quaternion matrices to dual quaternion matrices and exploring its applications in rigid body motion and  graph theory  is a research direction worth pursuing.

\item
Moore  is also known for proposing the general reciprocal in 1920 \cite{moore1920reciprocal}, which was redefined  by  Penrose \cite{penrose1955generalized} in 1955 and is nowadays called the Moore--Penrose inverse. Generalized inverse theory has since flourished; see \cite{benisrael2003generalized,wang2018generalized,cvetkovic2017algebraic, chen2024algebraic}.
Kyrchei employed the row--column determinant theory to compute various quaternion generalized inverses; see \cite[Chapters 4, 5]{Kyrchei2026}.
A natural question arises: can the $\mathrm{P}$-generalized determinant theory be used to compute quaternion generalized inverses effectively? This would directly overcome the fixed-order restriction of the row--column determinant, allowing any rule satisfying the theoretical requirements to be employed in computing generalized inverses.
\end{enumerate}

\section*{Funding}
This work was  supported by the National Natural Science Foundation of China [grant number 12371023].

\section*{Conflict of interest}
The authors declare no conflict of interest.

\section*{Author contributions}
\noindent\textbf{Jiale Gao:} Conceptualization, Methodology, Formal analysis, Investigation, Writing -- Original Draft, Writing -- Review \& Editing.

\noindent \textbf{Qing-Wen Wang:} Methodology, Funding acquisition, Supervision, Writing -- Review \& Editing.

\end{document}